\documentclass[10pt]{article}

\usepackage[a4paper,left=3cm,right=3cm,top=2cm,bottom=4cm,bindingoffset=5mm]{geometry}
\usepackage[utf8]{inputenc} 
\usepackage[T1]{fontenc}  
\usepackage{verbatim}   
\usepackage{amsmath}
\usepackage{amsthm}
\usepackage{amssymb}
\usepackage{amsfonts}
\usepackage{mathrsfs}
\usepackage{listings}
\usepackage{framed}
\usepackage{graphicx}
\usepackage{subfigure}
\usepackage{cite}
\usepackage{algorithmic}
\usepackage{algorithm}
\usepackage{color}
\usepackage{pgfplots}
\usepackage{tabularx}
\usepackage{multirow}
\usepackage[colorlinks,citecolor=blue]{hyperref}

\newtheorem{theorem}{Theorem}[section]
\newtheorem{corollary}[theorem]{Corollary}
\newtheorem{lemma}[theorem]{Lemma}
\newtheorem{assumption}{Assumption}
\newtheorem{definition}{Definition}[section]
\newtheorem{remark}[theorem]{Remark}

\newcommand{\N}{\mathbb{N}}
\newcommand{\R}{\mathbb{R}}

\newcommand{\eps}{\varepsilon}

\newcommand{\half}{\frac{1}{2}}

\newcommand{\intom}{\int\limits_{\Omega}}

\newcommand\ds{\,\mathrm{d}s}
\newcommand\dt{\,\mathrm{d}t}
\newcommand\dx{\,\mathrm{d}x}

\numberwithin{equation}{section}

\title{Optimal control of an evolution equation with non-smooth dissipation
	\thanks{This research was partially supported by the German Research Foundation (DFG) within the priority program "Non-smooth and
       Complementarity-based Distributed Parameter Systems: Simulation and Hierarchical
       Optimization" (SPP 1962) under
		grant number Wa 3626/3-1.}
	}

\author{
	Tobias Geiger%
	\thanks{University of Würzburg, Institute of Mathematics,
		Emil-Fischer-Str.\ 30, 97074 Würzburg, Germany; tobias.geiger@mathematik.uni-wuerzburg.de}
	\and
	Daniel Wachsmuth%
	\thanks{University of Würzburg, Institute of Mathematics,
		Emil-Fischer-Str.\ 30, 97074 Würzburg, Germany; daniel.wachsmuth@mathematik.uni-wuerzburg.de}
}

\begin{document}
\maketitle
\begin{abstract}

In the present work we study the optimal control of an evolution equation with non-smooth dissipation.  The solution mapping of this system is non-smooth and hence the analysis is quite challenging. Our approach is to regularize the dissipation via approximation by a smooth function. We derive optimality conditions for the corresponding smooth optimal control problem. Then we drive the regularization parameter to zero and obtain necessary optimality conditions for the original non-smooth problem. However, in this process we lose regularity of the adjoint variables.

\end{abstract}

{\small
\noindent{\bf Keywords: } optimal control, non-smooth state equation, necessary optimality conditions.
\\

\noindent{\bf AMS subject classification: }
49M20, 
65K10, 
90C30. 
\\
}

\section{Introduction}

We are interested in the optimal control of the following non-smooth evolution problem.
Let $\Omega \subset \R^d$ be a bounded Lipschitz domain, $d\in \N$,  and $I:=(0,T)$ an interval.
Let us define dissipation and energy functional by
\begin{equation}
\mathcal{D} \colon \; H^1_0(\Omega) \to \R, \quad \mathcal{D}(v) := \int_\Omega \vert v\vert + \frac{\sigma}{2} \vert \nabla v \vert^2\; \dx,
\end{equation}
\begin{equation}
\mathcal{E}\colon H^1_0(\Omega)\times L^2(\Omega) \to \R, \quad \mathcal{E}(z,g):= \int_\Omega \frac{1}{2} \vert \nabla z \vert^2 -z\cdot g\; \dx  .
\end{equation}
where $\sigma$ is a positive viscosity parameter.
The function $z$ is the state of the system, while the function $g$ acts as a distributed control.
Minimization of $\mathcal{E}(z(t),g(t)) + \mathcal{D}(\dot{z}(t))$ with respect to $z$ motivates the differential inclusion
\begin{equation}\label{eq003}
0 \in \partial\mathcal{D}(\dot{z}) + \partial_z \mathcal{E}(z,g),
\end{equation}
where $\partial$ denotes the convex subdifferential. We obtain
\begin{equation}\label{eq004}
0 \in \partial \vert \dot{z}(t,x) \vert - \Delta z(t,x) - \sigma \Delta \dot{z}(t,x) - g(t,x) \qquad \text{f.a.a.\@ } (t,x)\in I\times \Omega,
\end{equation}
where $g$ is the control and $z$ the state. The system is complemented by an initial condition $z(0)=z_0$.
Due to the appearance of the subdifferential, the evolution system is
inherently non-smooth. This makes the derivation of first-order necessary optimality conditions very challenging.

The non-smooth evolution system can be interpreted as a simplification of models appearing in applications.
Various different physical phenomena can be modelled by such non-smooth systems. This includes, e.g.,
electromagnetism, damage and crack propagation, and models with phase changes, see for instance the recent monograph \cite{MR2015}.
In order to focus on the impact of the non-smoothness of the model on the optimization, we decided to
study the simplified model with convex and quadratic energy.

Let us point out connections to other models studied in the literature.
Using a duality argument, we can rewrite the differential inclusion. To this end, let us introduce
\[
K:= \{ v \in H^1_0(\Omega)^*\mid  v \in L^2(\Omega),\; -1 \leq v \leq 1 \; \text{a.e. in}\; \Omega\},
\]
which is equal to the range of the subdifferential of the $L^1(\Omega)$-norm considered as a convex function on $H^1_0(\Omega)$.
As we will see in Lemma \ref{lemma: Normalenkegel}, the inclusion \eqref{eq004} is equivalent to
\begin{equation}\label{eq005}
\dot{z}(t) \in N_K(\Delta z(t) + \sigma\Delta \dot{z}(t) +g(t)) \qquad \text{f.a.a.\@ } t\in I.
 \end{equation}
Thus for $\sigma=0$ the inclusion can be considered as a sweeping process in the space $H^1_0(\Omega)$.
Let us emphasize two important properties of the set-valued mapping $z\mapsto N_K(\Delta z)$: first of all,
the images of this mapping are either unbounded or empty. And second, due to the results of \cite{CW17}
the set $K$ is not polyhedric in $H^1_0(\Omega)^*$.
At least one of these two properties is used in many works on optimal control of differential inclusions.
In addition, in both formulations \eqref{eq003} and \eqref{eq005} the arguments of the non-smooth mapping
contain the highest-order time or spatial derivative of $z$, which points to a lack of compactness in our system.
That is, the arguments of the non-smooth maps do not compactly depend on $z$ for sensible choices of function spaces.

Let us comment on available literature for control of non-smooth evolution systems.
Optimal control of parabolic variational inequalities of the type
$y_t-\Delta y + \beta(y)=u$
with $\beta$ a maximal monotone, set-valued operator were studied for instance in the monographs \cite{NeittaanmakiTiba1994,Tiba1990},
see also the recent contribution \cite{MSusu17}.
Optimal control problems of the sweeping process in finite-dimensions was studied in \cite{CHHM2012,CHHM2016}.
Recent works on optimal control of differential inclusions are \cite{MW2015,PRR2017}. There,
the set-valued map is assumed to have bounded images on bounded sets, an assumption that is not fulfilled in our setting.
The sweeping process is related to the so-called play operator, which is the solution map
of a rate-independent variational inequality. Optimal control problems
of the coupling of a play operator on $\mathbb R^n$ coupled with a ODE system was studied
in \cite{BK2016}, the coupling with a parabolic pde was investigated in \cite{M2017b}.
In \cite{GQ2011} the control of systems contain play operators on infinite time horizons was studied.
Due to the arguments above, all these results are not directly applicable to our setting.

To overcome the difficulties related to the non-smoothness of the system, we follow the popular approach of smoothing the state equation.
The regularization scheme is introduced in section \ref{sec3},
and its convergence properties are investigated in section \ref{section: Passing to limit in smooth equation}.
Passing to the limit with the regularization parameter, allows us to obtain a first-order system,
which is the main result of our paper in Theorem \ref{theorem: optimality system nonsmooth}.
It turns out that the appearing adjoint functions have rather low regularity.
The present work is strongly related to the earlier contribution  \cite{SWW2016}.
There, optimal control of an rate-independent system was studied, which corresponds to our problem with $\sigma=0$.
The positive parameter $\sigma>0$ enables us to prove stronger results than  \cite{SWW2016}. We comment on this at the end of Section \ref{sec74}, see Remark \ref{compare_to_SWW}.

\subsubsection*{Notation and function spaces}
We will work with the Lebesgue and Sobolev spaces $L^p(\Omega)$, $W^{1,p}(\Omega)$, $W^{1,p}_0(\Omega)$. In order to shorten the notation we define $H:= L^2(\Omega)$, $V:= H^1_0(\Omega)$, $V^*:= H^{-1}(\Omega):= (H^1_0(\Omega))^*$. We define the Laplace operator in a distributional sense

\begin{align*}
\Delta: V \to V^*: \ \langle \Delta u ,v \rangle_{V^*,V} := - \intom \nabla u(x) \cdot \nabla v(x) \dx
\end{align*}
The inner product on $V$ is defined by $(u,v)_{H^1_0(\Omega)}= (u,v)_V := \intom \nabla u(x) \cdot \nabla v(x) \dx $ and its induced norm by $\lVert u \rVert_V := \sqrt{(u,u)_V}$. Due to the zero boundary condition this norm is equivalent to the $H^1(\Omega)$-norm in the space $V$.
Since we are analyzing an evolution equation, we need Bochner spaces like $L^p(I,X)$ and $H^1(I,X)$, where $X$ is a real Banach space. The state equation of interest is equipped with a zero initial condition. Therefore we define the spaces
\[\begin{aligned}
H^1_\star(I,X)&:= \{ u \in H^1(I,X) \mid u(0) = 0  \}, \\
H^2_\star(I,X)&:= \{ u \in H^2(I,X) \mid u(0) = 0,\; \dot{u}(0) = 0  \}
,
\end{aligned}\]
where $\dot u$ denotes the weak derivative with respect to the time variable $t$.
Moreover, we work with the standard Hilbert triple $V \hookrightarrow H \cong H^* \hookrightarrow V^*$
induced by the $L^2$-inner product in order to use $L^2(\Omega)$-functions as elements of $V^*$.\\

\section{The non-smooth optimal control problem}

Let us first the define the notion of weak solutions of the differential inclusion \eqref{eq003}.

\begin{definition}
 A function $z\in H^1_\star(I,V)$ is called weak solution of \eqref{eq003} if and only if
 for almost all $t\in I$ it holds
 \[
 0 \in \partial \vert \dot{z}(t) \vert - \Delta z(t) - \sigma \Delta \dot{z}(t) - g(t) \qquad
 \text{ in } V^*.
 \]
 Here, $\partial \vert v\vert$ denotes the subdifferential of the $L^1(\Omega)$-norm
 with respect to the space $V$.
\end{definition}

The state equation is uniquely solvable and we have the following theorem, which will be proven in section \ref{section: Passing to limit in smooth equation}.

\begin{theorem} \label{theorem:  solution non-smooth}
	For all $g \in H^1_\star(I,V^*)$ there exists a unique solution	$z\in H^2_\star(I,V) $ of the non-smooth state equation \eqref{eq: non-smooth problem}.
\end{theorem}

Note that $g \in H^1_\star(I,V^*)$ and $z\in H^1_\star(I,V) $ include
the conditions $g(0)=0$ and $z(0)=0$. In view of the results derived below,
these conditions can be viewed as compatibility conditions at $t=0$.

As a conclusion of Theorem \ref{theorem: solution non-smooth} we can define a solution operator, which maps a control to the corresponding state.
\begin{equation}
\mathcal{S}: \; H^1_\star(I,L^2(\Omega)) \rightarrow H^1_\star(I,H^1_0(\Omega))\quad g \mapsto z .
\end{equation}
We now take a closer look at the subdifferential of the non-smooth part of $\mathcal{D}$, which is $\hat{\mathcal{D}}(v):= \lVert v \rVert_{L^1(\Omega)}$ and give characterizations of the state equation via cones. To this end we define
\begin{align*}
K&:= \{ v \in V^*\mid  v \in H,\; -1 \leq v \leq 1 \; \text{a.e. in}\; \Omega\},\\
N_K(v) &:= \{ w \in V \mid \langle w,\tilde{v}-v\rangle_{V,V^*} \leq 0 \; \forall \tilde{v} \in K \}, \\
\tilde{K}&:= \{ w \in V \mid \Delta w \in H,\; -1 \leq \Delta w \leq 1 \; \text{a.e. in}\; \Omega\} = \Delta^{-1}(K),\\
N^{\text{Hilbert}}_{\tilde{K}}(v) &:= \{ w \in V\mid (w,\tilde{v}-v)_{V} \leq 0 \; \forall \tilde{v} \in \tilde{K} \}.
\end{align*}

\begin{lemma} \label{lemma: subdiff L1}
	Let $v \in V$ and $f\in V^*$. Then it holds
	\begin{equation*}
	f \in \partial \hat{\mathcal{D}}(v) \quad \Leftrightarrow \quad f\in K  \text{ and } v \in N_K(f).
	\end{equation*}
	Moreover, $f\in \partial \hat{\mathcal{D}}(v)$ implies $f\in H$ and
	\begin{equation*}
	f(x) \in
	\begin{cases}
	\{ 1 \}& \text{ if }v(x) > 0 ,\\
	[-1,1] & \text{ if }v(x)= 0,\\
	\{ -1 \} & \text{ if }v(x) < 0,
	\end{cases}
	\end{equation*}
	i.e., $f(x)$ is in the subdifferential of the absolute value function evaluated at $v(x)$ for almost all $x\in \Omega$.
\end{lemma}

\begin{proof}
Let us denote by $\sigma_K$ be the support function and by $\delta_K$ the indicator function of $K$. Then
we have the following chain of equivalences
	\begin{align*}
	f \in \partial \hat{\mathcal{D}}(v) &\Leftrightarrow f \in \partial \sigma_K(v) \\
	&\Leftrightarrow f \in \partial \delta_K^*(v) \\
	&\Leftrightarrow v \in \partial \delta_K(f) \\
	&\Leftrightarrow \delta_K(h) \geq  \delta_K(f) + \langle v,h-f\rangle_{V^*,V}\quad \forall h \in V \\
	&\Leftrightarrow f\in K \; \text{and} \; 0 \geq \langle v,h-f \rangle_{V^*,V} \quad \forall h \in K \\
	&\Leftrightarrow f\in K \; \text{and} \; v \in N_K(f)
	,
	\end{align*}
which proves the first part of the lemma.

Let now $f\in \partial \hat{\mathcal{D}}(v)$ be satisfied.
We already proved in the first part $f\in H$ and $-1 \leq f \leq 1$ a.e.\@ on $\Omega$.
It remains to prove

	\begin{equation*}
	f(x) \in
	\begin{cases}
	\{ 1 \}&\text{for}\; v(x) > 0 ,\\
	\{ -1 \} &\text{for}\; v(x) < 0.
	\end{cases}
	\end{equation*}
Assume there exists a set $M \subset \{ x \in \Omega :\ \vert v(x) \vert >0 \}$ with positive measure such that $\vert f(x) \vert< 1$ a.e.\@ on $M$. Hence there is an $\eps >0$ and a set $M_\eps \subset M$ with $\vert M_\eps \vert >0$ such that $\vert f(x) \vert < 1-\eps$ for a.a.\@ $x\in M_\eps$. We obtain the existence of a $\delta >0$ and a subset $A_\delta \subset M_\eps$ with $\vert A_\delta \vert > 0$ and $v(x)> \delta$ a.e.\@ on $A_\delta$.
	\\
 Due to the positive homogeneity of $\hat{\mathcal{D}}$ we have $  f\in \partial \hat{\mathcal{D}}(v) \Leftrightarrow \hat{\mathcal{D}}(v)= \langle f,v \rangle_{V^*,V}$, see e.g. \cite[Lemma 1.3.1]{MR2015} and we obtain

	\begin{align*}
	\hat{\mathcal{D}}(v)  &= \langle f,v \rangle_{V^*,V}
	= \int\limits_{\Omega} f(x)v(x)\, \dx
	= \int\limits_M fv\, \dx +   \int\limits_{\{ v(x)=0 \}} fv\, \dx
	= \int\limits_{M_\eps} fv\, \dx + \int\limits_{M\setminus M_\eps} fv\, \dx  \\
	&\leq (1-\eps) \int\limits_{M_\eps} \vert v \vert\, \dx + \int\limits_{M\setminus M_\eps} \vert v \vert\, \dx
	= \int\limits_{\Omega} \vert v(x) \vert\, \dx - \eps \int\limits_{M_\eps} \vert v(x) \vert\, \dx \\
	&\leq \lVert v \rVert_{L^1(\Omega)} - \eps \int\limits_{A_\delta} \vert v(x) \vert\, \dx
	\leq \lVert v \rVert_{L^1(\Omega)} - \eps \delta
	<  \lVert v \rVert_{L^1(\Omega)} = \hat{\mathcal{D}}(v),
	\end{align*}
	which is a contradiction.
\end{proof}

Using this lemma one can easily verify the following characterizations of the state equation.

\begin{lemma} \label{lemma: Normalenkegel}
	Let $z \in H^1_\star(I,V)$ and $g\in H^1(I,V^*)$ be given.
	Then the following statements are equivalent.

	\begin{enumerate}
		\item $z$ is a weak solution of \eqref{eq003} to $g$, i.e.,
		\begin{equation*}
		0 \in \partial \vert \dot{z}(t) \vert - \Delta z(t) - \sigma \Delta \dot{z}(t) - g(t) \quad \text{in } V^* \; \text{a.e.\@ on I.}
		\end{equation*}
		\item \begin{equation}
		\dot{z} \in N^{\text{Hilbert}}_{\tilde{K}}(-z-\sigma\dot{z}-\Delta^{-1}g) \quad \text{in } V\; \text{a.e.\@ on I.}
		\end{equation}
		\item \begin{equation}
		\dot{z} \in N_K(\Delta z + \sigma\Delta \dot{z} +g) \quad \text{in } V\; \text{a.e.\@ on I.} \label{eq: normal cone}
		\end{equation}
	\end{enumerate}

\end{lemma}

Evolution inclusion using the normal cone are known from other problems like the sweeping process, and optimal control problems of this process are analyzed, e.g., in \cite{CHHM2012,CHHM2016}. An important difference is that in our case the time derivative of the state
as well as $\Delta z$ are arguments of the normal cone mapping.

In the next lemma we prove a continuity property of the solution operator $\mathcal{S}$. This lemma as well as the proof are from \cite[Lemma 3.4]{SWW2016}

\begin{lemma} \label{lemma: convergence of solution operator}
	Let $(g_n)_{n\in\N} \in H^1_\star(I;H)$ be a sequence with $g_n \rightharpoonup g$ in $H^1(I,H)$. Then
	$\mathcal{S}(g_n) \rightarrow \mathcal{S}(g)$ in $H^1_\star(I,V)$ and in $\mathcal{C}(\bar{I},V)$. Moreover, $\mathcal{S}$ is Lipschitz continuous from $L^2(I,V^*)$ to $H^1_\star(I,V)$.
\end{lemma}

\begin{proof}
Let us denote $z:=\mathcal{S}(g)$, $z_n:=\mathcal{S}(g_n)$.
Due to the continuity of the embedding  $H^1(I,H)\hookrightarrow \mathcal{C}(\bar{I},V)$ we have $g(0)= 0$. Testing \eqref{eq: normal cone} for $z$ with $\sigma \Delta \dot{z}_n + \Delta z_n + g_n$ and for $z_n$ with $\sigma \Delta \dot{z} +\Delta z + g$ we obtain by adding both inequalities and integrating from $0$ to $t$

\begin{equation*}
\frac{1}{2} \lVert z(t) - z_n(t) \rVert_{V}^2  + \sigma \lVert \dot{z}-\dot{z}_n \rVert_{L^2(0,t;V)}^2 \leq \int_0^t \langle \dot{z}_n - \dot{z},g_n-g\rangle_{V,V^*} \ds.
\end{equation*}
Taking the supremum with respect to $t$ yields
\begin{equation*}
\frac{1}{2} \lVert z - z_n \rVert_{C(\bar{I},V)}^2  + \sigma \lVert \dot{z}-\dot{z}_n \rVert_{L^2(I,V)}^2 \leq \lVert \dot{z}-\dot{z}_n \rVert_{L^2(I,V)}\lVert g_n -g \rVert_{L^2(I,V^*)}.
\end{equation*}
Young's  inequality $ab \leq \frac{\sigma}{2}a^2 +  \frac{1}{2\sigma}b^2$ gives

\begin{equation} \label{proof: continuity S 1}
\frac{1}{2} \lVert z - z_n \rVert_{C(\bar{I},V)}^2 + \frac{\sigma}{2} \lVert \dot{z}-\dot{z}_n \rVert_{L^2(I,V)}^2 \leq \frac{1}{2\sigma} \lVert g_n -g \rVert_{L^2(I,V^*)}^2.
\end{equation}

From the Aubin-Lions lemma, see, e.g., \cite{A1963,L1969}, we know that the embedding $H^1(I,H) \hookrightarrow L^2(I,V^*)$ is compact, which proves the assertion.
\end{proof}

\begin{remark}
	The proof shows explicitly $\mathcal{S}(g_n) \to \mathcal{S}(g)$ in $\mathcal{C}(\bar{I},V)$. However, this is also a consequence of the continuity of the embedding $H^1_\star(I,V)\hookrightarrow \mathcal{C}(\bar{I},V) $.
\end{remark}

We will use the previous lemma to show existence of solutions of the optimal control problem \eqref{eq: non-smooth problem} below.

In order to formulate the optimal control problem, we take two functions
\begin{align}
&j_1:\; L^2(I;V) \to \R,  \\
&j_2:\; V \to \R,
\end{align}
which we assume to be continuously Fr\'{e}chet differentiable and bounded from below.
The objective function is given by
\begin{align*}
J:\;  H^1(I,V) \times H^1(I,H) \; &\to \; \R, \notag \\
(z,g) &\mapsto J(z,g) := j_1(z) + j_2(z(T)) + \frac{1}{2} \lVert g \rVert_{H^1(I;H)}^2.
\end{align*}
In the sequel we will study the following optimal control problem
\begin{align}
&\min J(z,g) \qquad \text{with respect to } (z,g)\in H^1(I,V) \times H^1(I,H) \notag  \\
&\text{subject to}
\left\{\!\begin{aligned} &0 \in \partial \vert \dot{z}(t) \vert - \Delta z(t) - \sigma \Delta \dot{z}(t) - g(t) \qquad \text{ in }  V^* \text{ for a.a.\@ } t\in I, \\
&g(0) = 0,\; z(0) = 0.
\end{aligned}\right.
\label{eq: non-smooth problem}
\tag{P}
\end{align}

\begin{theorem} \label{lemma: global solution nonsmooth}
	There exists a solution of the optimal control problem \eqref{eq: non-smooth problem}.
\end{theorem}
\begin{proof}
The proof uses the standard direct method.
Let  $ (g_n,z_n)_n \in H^1_\star(I,H)\times H^1_\star(I,V)$ be a minimizing sequence. In particular, $z_n = \mathcal{S}(g_n)$ holds. Since $j_1,j_2$ are bounded from below, we get that $\half \lVert g_n \rVert^2_{H^1(I,V^*)}$ is bounded, and there exists $g\in H^1_\star(I,H)$ such that $g_n \rightharpoonup g $ after possibly extracting a subsequence.
Lemma \ref{lemma: convergence of solution operator} shows that $z_n \to z =\mathcal{S}(g)$ in $H^1_\star(I,V)$. Since $j_1,j_2$ are assumed to be continuous, and $\lVert \cdot \rVert_{H^1(I,H)}^2$ is  weakly lower semicontinuous, we get $j_1(z_n)\to j_1(z)$, $j_2(z_n(T))\to j_2(z(T))$, $\lVert g \rVert^2_{H^1(I,V^*)} \leq \liminf\limits_{n\to \infty} \lVert g_n \rVert^2_{H^1(I,V^*)}$. Hence
it follows $J(\mathcal{S}(g),g) = J(z,g)\leq \liminf\limits_{n\to \infty} J(z_n,g_n)$, i.e., $(z,g)$ solves the state equation and is (globally) optimal.
\end{proof}

We are interested in proving necessary optimality conditions for the non-smooth optimal control
problem \eqref{eq: non-smooth problem}.

%
%
%

\section{The regularized state equation}
\label{sec3}

In this section we approximate the non-smooth part of the dissipation function and analyze the resulting equation.

\subsection{Smooth approximation of the dissipation} \label{subsection: Approximation of absolute value}

The function $H^1_0(\Omega) \ni v \mapsto \lVert v \rVert_{L^1(\Omega)} $ is non-smooth, which makes the state equation quite uncomfortable to deal with. For this reason we will replace the absolute value function in the $L^1$-norm by a smooth approximation function. The idea how to choose the approximation is from \cite[section 4.1]{SWW2016}.

Let $\rho > 0 $ be a positive parameter and define a family of functions
\begin{equation*}
\vert \cdot \vert_\rho: \; \R \rightarrow \R, \qquad x \to \vert x \vert_\rho
.
\end{equation*}
The family $\{\vert \cdot \vert_\rho \}_{\rho>0}$ should satisfy some properties.

\begin{assumption} \label{assump: smooth approx}
	Let $\rho >0$. We assume for the family $\{\vert \cdot \vert_\rho \}_{\rho>0}$ the following properties.
	\begin{enumerate}
		\item $\vert \cdot \vert_\rho $ is in $\mathcal{C}^2(\R,\R)$.
		\item $\vert \cdot \vert_\rho $ is convex.
		\item $\vert \cdot \vert_\rho$ is an even function, i.e. $\vert v \vert_\rho=\vert -v \vert_\rho$ for all $v\in \R$.
		\item $\vert v \vert_\rho = \vert v \vert$ for all $v\in \R$ with $\vert v \vert \geq \rho$.
		\item $\vert v \vert_\rho'' \leq \frac{2}{\rho}$ for all $v \in \R$.
		\item The second derivatives $\vert \cdot \vert_\rho'' $ are Lipschitz continuous with Lipschitz constant $\frac{2}{\rho^2}$.
		\item $\vert v \vert_{\rho_1} \leq \vert v \vert_{\rho_2}  $ for all $v\in \R$ and $\rho_1 \leq \rho_2$.
		\item For all $\rho_1,\rho_2 > 0$ and $v\in \R$ holds $\big\vert \vert v \vert_{\rho_1}- \vert v \vert_{\rho_2}\big\vert \leq \, \vert \rho_1 - \rho_2 \vert$. \label{assumption, estimate two rhos}
	\end{enumerate}
\end{assumption}

\begin{lemma} \label{lemma: properties approx}
	Let the family $\{ \vert \cdot \vert \}_{\rho >0}$ satisfy Assumption \eqref{assump: smooth approx}. Then it holds for all $v\in \R$:

	\begin{minipage}{\textwidth}
		\begin{enumerate}
			\item $\vert v \vert_\rho' \in [-1,1] $,
			\item $\vert v \vert_\rho'' \geq 0$,
			\item $\vert v \vert \leq \vert v \vert_\rho \leq \vert v \vert + \rho$,
			\item $\vert v \vert_\rho' v \geq \vert v \vert - \rho$,
			\item $\vert v \vert_\rho'' v^2 \leq 2\rho$.
		\end{enumerate}
	\end{minipage}
\end{lemma}

\begin{proof}
1. and 2. follow immediately from convexity. 3. and 4. can be found in \cite{SWW2016}.
It remains to prove 5.:  Due to $\vert v \vert_\rho'' \leq \frac{2}{\rho^2} $ we have
		$
		\vert v \vert_\rho'' v^2 \leq \frac{2}{\rho} \rho^2 \leq 2\rho.
		$
\end{proof}

A function satisfying Assumption \ref{assump: smooth approx} exists. An example is
\begin{equation}
\vert \cdot \vert_\rho : \R \to \R \qquad v\mapsto
\begin{cases}
\vert v \vert & \vert v \vert \geq \rho, \\
\frac{1}{3}\rho + \frac{1}{\rho^2}v^2(\rho - \frac{1}{3}\vert v \vert) & \vert v \vert \leq \rho.
\end{cases}
\end{equation}


%
%

\subsection{Existence and uniqueness of solutions for the smooth state equation} \label{subsection: Ex and un for smooth state equation}

In this section we are going to modify the dissipation function $\mathcal{D}$ by using the family $\{ \vert \cdot \vert_\rho \}$. This section is oriented on \cite[Section 4.2]{SWW2016}.
Consider the modified dissipation function
\begin{equation}
\mathcal{D}_\rho : H^1_0(\Omega) \to \R, \; \mathcal{D}_\rho(v) := \int_\Omega \vert v \vert_\rho + \frac{\sigma}{2}\vert \nabla v \vert^2 \dx
\end{equation}
for an arbitrary $\rho > 0$. Using this regularized dissipation instead of $\mathcal D$ leads to the inclusion
\begin{equation*}
0 \in \partial \mathcal{D}_\rho(\dot{z}) +\partial_z \mathcal{E}(z,g).
\end{equation*}
The regularized dissipation $\mathcal{D}_\rho$ is differentiable, which means in particular that the inclusion is actually an equation. Furthermore, we require the initial condition $z(0)=0$. We obtain the following regularized state equation

\begin{subequations} \label{eq:state glatt}
	\begin{align}
	\vert \dot{z}(t)\vert'_\rho - \sigma \Delta \dot{z}(t) - \Delta z(t) &= g(t) &&\text{in } V^* \text{ f.a.a.\@ } t\in I, \\
	z(0) &= 0. &&
	\end{align}
\end{subequations}

The first step in analyzing the regularized state equation is to show that for every control $g \in L^2(I,V^*)$ there exists a unique state $z\in H^1_\star(I,V)$, which solves equation \eqref{eq:state glatt}.
Using the substitution $w:=\dot{z}$ we can reformulate equation \eqref{eq:state glatt} in the following way.
\begin{subequations} \label{eq: state glatt 2}
	\begin{align}
	\dot{z} &= w  &&\text{in }  V \text{ a.e.\@ on }I, \\
	-\sigma \Delta w + \vert w \vert'_\rho &= \Delta z + g &&\text{in }  V \text{ a.e.\@ on }I, \\
	z(0)&=0. &&
	\end{align}
\end{subequations}
In order to solve the system \eqref{eq: state glatt 2} we first analyze the equation

\begin{equation} \label{equation: state mon operator}
-\sigma \Delta w + \vert w \vert'_\rho = v \quad \text{in } V^*,
\end{equation}
where $v\in V^*$ is arbitrary. The operator

\begin{equation} \label{equation: def mon operator}
A_\rho : V \to V^*,\; A_\rho(w):= -\sigma \Delta w + \vert w \vert'_\rho
\end{equation}
is strongly monotone and hemi-continuous with

\begin{equation*}
\langle Aw_1 - Aw_2 , w_1 - w_2 \rangle_{V^*,V} \geq \sigma \lVert w_1 - w_2 \rVert_V^2 \qquad \forall \, w_1,w_2 \in V.
\end{equation*}
Hence equation \eqref{equation: state mon operator} is uniquely solvable and its solution operator

\begin{equation*}
T_\rho : V^* \to V,\quad T_\rho(v) := A_\rho^{-1}
\end{equation*}
is Lipschitz continuous with Lipschitz constant $\frac{1}{\sigma}$.


Using the operator $T_\rho$, we can reformulate the regularized state equation as an initial value problem in the Banach space $H^1_0(\Omega)$,

\begin{equation} \label{equation: ODE in H-1}
\begin{cases}
\dot{z}(t) = T_\rho \bigl(g(t) + \Delta z(t)\bigr) &\text{in } H^1_0(\Omega) \quad \text{f.a.a.\@ }t \in I , \\
z(0)= 0   &\text{in } H^1_0(\Omega).
\end{cases}
\end{equation}
This initial value problem is uniquely solvable due to the Lipschitz continuity of $T_\rho$, and the solution operator

	\begin{equation}
	\mathcal{S}_\rho: L^2(I,V^*) \to H^1_\star(I,V), \quad g \mapsto z
	\end{equation}
is continuous, see \cite[Satz 1.3]{GGZ1974}.

\subsection{Differentiability of the solution operator and Lipschitz estimates} \label{section: Diffbarkeit smooth solution operator}

The next step is to prove the Fr\'{e}chet differentiability of the solution operator $\mathcal{S}_\rho$ and to formulate an equation, which is solved by its derivative. This will be important for finding the optimality conditions, as it allows us to  use the reduced functional.

In order to prove differentiability of $\mathcal{S}_\rho$, we first show that $T_\rho$ is differentiable.

\begin{theorem} \label{theorem: Trho diffbar}
	Let $\rho >0$. The operator $T_\rho:V^*\to V$ is Fr\'{e}chet differentiable. Let $v,h \in V^*$ be given  and define $w:= T_\rho(v)$. Let $y\in H^1_0(\Omega)$ be the unique weak solution of the equation
	\begin{equation} \label{equation: derivative solution operator}
	-\sigma \Delta y + \vert w \vert_\rho''y = h \qquad \text{in } V^*.
	\end{equation}
	Then it holds $ T_\rho'(v)h = y $ and
	\begin{equation} \label{equation: continuous dependence of derivative}
	\lVert y \rVert_{V} = \lVert T_\rho'(v)h \rVert_{V} \leq \frac{1}{\sigma} \lVert h \rVert_{V^*}.
	\end{equation}
\end{theorem}
\begin{proof}
	Equation \eqref{equation: derivative solution operator} is uniquely solvable in $V$
	due to the Lax-Milgram theorem since $\vert w \vert_\rho''\in L^\infty(\Omega)$ is a nonnegative
	coefficient.
	In addition, it holds $\lVert y \rVert_V \leq \frac{1}{\sigma} \lVert h \rVert_{V^*}$.

	In order to show the Fr\'{e}chet differentiability we investigate the remainder $r_h :=  T_\rho(v+h) - w - y$.
	By definition of $T_\rho(v+h)$, $T_\rho(v)$, and $y$ we have
	\begin{align*}
	- \sigma \Delta T_\rho(v+h) + \vert T_\rho(v+h) \vert_\rho' &= v+h &&\text{in }V^*, \\
	- \sigma \Delta w + \vert w \vert_\rho' &= v &&\text{in } V^*, \\
	-\sigma \Delta y + \vert w \vert_\rho''y &= h &&\text{in } V^*.
	\end{align*}
	Subtracting the second and third from the first equation, adding and subtracting $\vert w \vert_\rho''\bigl( T_\rho(v+h) - w \bigr)$ yield
	\begin{equation}
	- \sigma \Delta r_h + \vert w \vert_\rho'' r_h  = -\bigl( \vert T_\rho(v+h) \vert_\rho' - \vert w \vert_\rho' - \vert w \vert_\rho''(T_\rho(v+h) - w) \bigr).
	\end{equation}
	Lax-Milgram implies that $r_h$ is the unique weak solution of this equation, and we get the estimate
	\begin{equation*}
	\lVert r_h \rVert_{V}  \leq \frac{1}{\sigma} \Big\lVert \vert T_\rho(v+h) \vert_\rho' - \vert w \vert_\rho' - \vert w \vert_\rho''(T_\rho(v+h) - w) \Big\rVert_{H} \label{proof: Trho diffbar 1}.
	\end{equation*}

The embedding theorems for Sobolev spaces give us the existence of  $p>2$ such that $V\hookrightarrow L^p(\Omega)$.
Due to the boundedness of $\vert \cdot \vert_\rho'$ the Nemytskij operator of this mapping is Fr\'{e}chet differentiable from $L^p(\Omega)$ to $H$.
This shows
\[
\lVert r_h \rVert_{V} = o\big(\lVert T_\rho(v+h) - w\|_V\big) = o\big(\|h\|_{V^*}\big),
\]
which proves
the Fr\'{e}chet differentiability of $T_\rho$.
\end{proof}

\begin{theorem} \label{theorem: S_rho diffbar}
	Let $\rho >0$, $2\leq p< \infty$, and  $1 \leq q <p$ be given. Then the operator $\mathcal{S}_\rho$ is Fr\'{e}chet differentiable as a mapping from $L^p(I,V^*)$ to $W^{1,q}_\star(I,V)$. For $g,h \in L^p(I,V^*) $ define $z := \mathcal{S}_\rho(g),\; \zeta:= \mathcal{S}'_\rho(g)h$. Then $\zeta$ is the unique solution of the system
	\begin{subequations} \label{equation: Ableitung Lösungsoperator}
		\begin{align}
		\dot{\zeta} &= \omega &&\text{in } V\; \text{a.e.\@ on } I,\\
		-\sigma \Delta \omega + \vert \dot{z}\vert''_\rho \omega &= \Delta \zeta + h &&\text{in } V^*\; \text{a.e.\@ on } I,\\
		\zeta(0)&=0 &&\text{in } V.
		\end{align}
	\end{subequations}

\end{theorem}

\begin{proof}
This can be proven following the lines of the proof of
\cite[Theorem 4.4]{SWW2016}.
\end{proof}

Later we will consider controls $g$ in the space $H^1_\star(I,H)$. In this case we have the following result.

\begin{corollary}
	Let $g\in H^1_\star(I,H)$ and $1\leq q < \infty$. Then $S_\rho$ is Fr\'{e}chet differentiable as a mapping from $H^1_\star(I,H)$ to $W^{1,q}_\star(I,V)$.
\end{corollary}

\begin{proof}
	Choose $p$ such that $q < p < \infty$ holds. The following embeddings are continuous.
	\begin{equation*}
	H^1_\star(I,H)\hookrightarrow \mathcal{C}(\bar{I},H) \hookrightarrow L^p(I,V^*).
	\end{equation*}
	Therefore the claim follows from Theorem \eqref{theorem: S_rho diffbar}.
\end{proof}

In the next lemma we show a Lipschitz property for $S_\rho$. This lemma is a stronger version of \cite[Lemma 4.5]{SWW2016}.

\begin{lemma} \label{lemma: Lipschitz smooth solution operator}
	Let $2\leq p \leq \infty$ and $g_1,g_2 \in L^p(I,V^*)$ be given. For $i=1,2$ define $z_i:= \mathcal{S}_\rho(g_i)$. Then it holds $z_i \in W^{1,p}_\star(I,V)$.
	In addition, we have for almost all $t\in I$

	\begin{equation}
	\lVert \dot{z}_1(t) - \dot{z}_2(t) \rVert_{V} \leq \frac{1}{\sigma} \big\lVert g_1(t) - g_2(t) \big\rVert_{V^*} + \frac{1}{\sigma^2}e^{\frac{1}{\sigma}t}\big\lVert g_1 -g_2 \big\rVert_{L^1(0,t;V^*)}. \label{eq: Lipschitz Srho}
	\end{equation}

\end{lemma}

\begin{proof}

We obtain with the Lipschitz continuity  of $T_\rho$ for a.a. $t\in I$

	\begin{align}
	\lVert  \dot{z}_1(t) &- \dot{z}_2(t)  \rVert_{V} \notag
	= \lVert  T_\rho(g_1(t) + \Delta z_1(t)) - T_\rho(g_2(t)+ \Delta z_2(t))   \rVert_{V} \notag \\
	&\leq \frac{1}{\sigma} \lVert  g_1(t) - g_2(t)   \rVert_{V^*} + \frac{1}{\sigma} \lVert  z_1(t) -  z_2(t)   \rVert_{V} \label{eq: proof Lipschitz Srho 1}
	.
	\end{align}
By integrating \eqref{equation: ODE in H-1} from $0$ to $t$, we obtain
	\begin{align*}
	\lVert z_1(t) -  z_2(t)   \rVert_{V}
	&=\Big\lVert \int\limits_{0}^{t} T_\rho(g_1(s)+\Delta z_1(s) ) - T_\rho(g_2(s)+\Delta z_2(s) )\ds \Big\rVert_{V}  \\
	&\leq  \frac{1}{\sigma} \Big\lVert  g_1 - g_2 \Big\rVert_{L^1(0,t;V^*)} + \frac{1}{\sigma} \int\limits_{0}^{t} \Big\lVert z_1(s) - z_2(s) \Big\rVert_{V} \ds
	\end{align*}
We apply Gronwall's inequality and obtain

	\begin{equation}
	\lVert z_1(t) - z_2(t) \rVert_{V} \leq  \frac{1}{\sigma^2}e^{\frac{1}{\sigma}t}\big\lVert g_1 -g_2 \big\rVert_{L^1(0,t;V^*)}. \label{eq: proof Lipschitz Srho 2}
	\end{equation}
	Combining \eqref{eq: proof Lipschitz Srho 1} and \eqref{eq: proof Lipschitz Srho 2} we get the asserted inequality.
	Moreover, choosing $g_2=0$ in \eqref{eq: Lipschitz Srho} gives

	\begin{equation*}
	\lVert \dot{z}_1(t)\rVert_{V} \leq \frac{1}{\sigma} \big\lVert g_1(t)  \big\rVert_{V^*} + \frac{1}{\sigma^2}e^{\frac{1}{\sigma}T}\big\lVert g_1  \big\rVert_{L^1(I,V^*)},
	\end{equation*}
hence $z_1 \in W^{1,p}_\star(I,V)$ holds.
\end{proof}

\subsection{Higher regularity of the state and a-priori estimates}

Let us now prove some a-priori estimates for the state $z$ of the regularized equation.
We will also prove higher regularity results for the state $z$ in space and time under some assumptions on the domain $\Omega$ and on the control $g$. The next lemma is from \cite[Lemma 4.6]{SWW2016}.

\begin{lemma} \label{lemma: H2 time}
	Let $\rho > 0$ and $g\in H^1_\star(I,V^*)$ be given. Define $z:= \mathcal{S}_\rho(g)$. Then it holds $z \in H^2(I,V)$ and

	\begin{equation*}
	\lVert \ddot{z}(t) \rVert_{V} \leq \frac{1}{\sigma} \big\lVert \dot{g}(t) + \Delta \dot{z}(t) \big\rVert_{V^*}\qquad \text{a.e.\@ on } I.
	\end{equation*}
Moreover,
there is a constant $C>0$ independent of $\sigma$, $\rho$ such that
	\begin{equation}
	\lVert \dot{z}(T) \rVert_{V}^2  + \lVert \dot{z} \rVert_{L^2(I,V)}^2
	\leq C\cdot \Bigl[\bigl(\frac{2\rho}{\sigma} + 1\bigr)\cdot \vert \Omega \vert + \frac{1}{\sigma} \lVert \dot{g} \rVert_{L^2(I,V^*)}^2 \Bigr]
	\end{equation}
	and
	\begin{equation}
	\lVert \dot{z}(0)\rVert_{V} \leq \frac{\rho}{\sigma} \vert \Omega \vert. \label{eq: estimate dotz_0}
	\end{equation}
are satisfied.
\end{lemma}

\begin{proof}
The proof is the same as  \cite[Proof of Lemma 4.6]{SWW2016}, except that we have $\sigma \Delta \dot{z}$ in \eqref{eq:state glatt} instead of $\rho \Delta \dot{z}$.
\end{proof}

We now turn our focus on regularity results in space. In order to prove higher regularity in space for the state $z$ we need to assume higher regularity in space for the control $g$, i.e. $g\in H^1_\star(I,H)$. We first show that for a fixed $t\in I$ the function $\dot{z}(t)$ solves an elliptic PDE.

\begin{lemma}\label{lemma: regularity space}

	Let $\rho >0$ and $g \in H^1_\star(I,H)$. Let further $z := \mathcal{S}_\rho g \in H^1_\star(I,V)$ be the unique solution of the regularized state equation \eqref{eq:state glatt}.  Then
	it holds $\Delta z(t),\Delta \dot{z}(t) \in H$ f.a.a. $t\in I$.
	In addition, we have the estimates
	\begin{align}
	\lVert \Delta z \rVert_{C(\bar{I},H)} &\leq \sqrt{\frac{1}{2\sigma}}  \big\lVert g - \vert \dot{z} \vert_\rho'  \big\rVert_{L^2(I,H)}, \label{estimate regularity space 1}  \\
	\lVert \Delta \dot{z} \rVert_{L^2(I,H)} &\leq \frac{1}{\sigma} \big\lVert g - \vert \dot{z} \vert_\rho' \big\rVert_{L^2(I,H)}. \label{estimate regularity space 2}
	\end{align}
\end{lemma}

\begin{proof}
Let us set $f:= \frac{1}{\sigma}\bigl( g - \vert \dot{z} \vert_\rho'  \bigr)$, hence  $f \in L^2(I,H)$.
The initial value problem
	\begin{equation} \label{proof regularity space1}
	\begin{cases}
	\dot v(t) + \frac{1}{\sigma}v(t) = f(t) &\text{in } H\;\text{for a.a.}\; t\in I, \\
	v(0) = 0  &\text{in } H
	\end{cases}
	\end{equation}
	has a unique solution $v \in H^1(I,H) $ due to \cite[Satz 1.3]{GGZ1974}.
We multiply \eqref{proof regularity space1} with $\dot{v}(t)$ and integrate over $(0,t)\times \Omega$. This yields
	\begin{equation*}
	\int\limits_{0}^{t} \bigl(\dot{v}(t) , \dot{v}(t)  \bigr)_H\dt + \frac{1}{\sigma} \int\limits_{0}^{t} \bigl(v(t) , \dot{v}(t)  \bigr)_{H}\dt = \int\limits_{0}^{t} \bigl(f(t) , \dot{v}(t)  \bigr)_{H}\dt
	\leq \lVert f \rVert_{L^2(I,H)} \cdot \lVert \dot{v} \rVert_{L^2(0,t;H)}.
	\end{equation*}
	Therefore we obtain
	\begin{equation*}
	\lVert \dot{v} \rVert_{L^2(I,H)}^2 + \frac{1}{2\sigma} \lVert v \rVert_{C(\bar{I},H)}^2 \leq  \lVert f \rVert_{L^2(I,H)} \cdot \lVert \dot{v} \rVert_{L^2(I,H)} \leq \frac{1}{4} \lVert f \rVert_{L^2(I,H)}^2 + \lVert \dot{v}\rVert_{L^2(I,H)},
	\end{equation*}
which implies the two inequalities
	\begin{align}
	\lVert \dot{v} \rVert_{L^2(I,H)} &\leq \lVert f \rVert_{L^2(I,H)}, \label{proof regularity space2} \\
	\lVert v \rVert_{C(\bar{I},H)} &\leq \sqrt{\frac{\sigma}{2}}  \lVert f \rVert_{L^2(I,H)}. \label{proof regularity space3}
	\end{align}
By construction of $f$, we have

	\begin{equation}
	\begin{cases}
	- \Delta \dot{z}(t) - \frac{1}{\sigma}\Delta z(t) = f(t) &\text{in } V^*\;\text{for a.a.}\; t\in I, \\
	-\Delta z(0) = 0  &\text{in } V^*.
	\end{cases}
	\end{equation}
Since this initial value problem is uniquely solvable, it follow
	$-\Delta z(t) = v(t) \in H$ f.a.a. $t\in I$.
\end{proof}

Using the previous lemma we can apply several known results about higher regularity. We only mention one of them here.

\begin{corollary} \label{cor: regularity in the interior} (Regularity in the interior) \\
	Let $\rho >0$ and $g \in H^1_\star(I,H)$. Let further $z := \mathcal{S}_\rho g \in H^1_\star(I,V)$ and define $v\in H^1(I,H),f\in L^2(I,H)$ as in the previous lemma.
	Let $\Omega_0\subset \Omega$ be an open set compactly contained in $\Omega$.
	Then it holds $z \in H^1_\star(I,H^2(\Omega_0))$, and there exists a constant $C$ independent of $\rho$ such that
	\begin{align*}
	\lVert z \rVert_{C(\bar I,H^2(\Omega_0))}
	&\leq C\bigl( 1 + \lVert g \rVert_{L^2(I,H)}+\lVert z \rVert_{C(\bar I,V} \bigr) , \\
	\lVert \dot z \rVert_{L^2(I,H^2(\Omega_0))}
	&\leq C\bigl( 1 + \lVert g \rVert_{L^2(I,H)}+\lVert \dot z \rVert_{L^2(I,V} \bigr) , \\
	\lVert \ddot{z} \rVert_{L^2(I,H^2(\Omega_0))}
	&\leq C \bigl( \lVert \ddot{z} \rVert_{L^2(I,V)} +  \lVert \Delta\dot{z}\rVert_{L^2(I,H)} + \lVert \dot{g} \rVert_{L^2(I,H)} + \frac{1}{ \rho} \lVert \ddot{z}\rVert_{L^2(I,H)}\bigr).
	\end{align*}
\end{corollary}

\begin{proof}
	Applying the well known theorem about regularity in the interior for elliptic PDEs on Lipschitz domains, which can be found, e.g., in \cite[Section 6.3]{E2010}, gives us
	the existence of $C>0$ such that
	\[
	\lVert u \rVert_{H^2(\Omega_0)}
	\leq C\bigl( \lVert u \rVert_V+ \lVert \Delta u \rVert_H\bigr)
	\]
	for all $u\in V$ with $\Delta u\in H$.
Then the claimed estimates of $z$ and $\dot z$ are a consequence of the previous Lemma \ref{lemma: regularity space}.

It remains to prove the estimate of $\ddot z$.
Here, we will use the function $v$ as defined in the previous proof.
Let us choose $h\in \R$ such that $t+ h \in I$. Then
	\[
	-\Delta \big[ \dot{z}(t+h) - \dot{z}(t) \big] = \dot{v}(t+h) - \dot v(t) \quad \text{for a.a.}\; t\in I \; \text{in} \; V^*.\\
	\]
	Recall from the proof of Lemma \ref{lemma: regularity space}, that $\dot{v}+ \frac{1}{\sigma}v = f= \frac{1}{\sigma}\bigl( g - \vert \dot{z} \vert_\rho'  \bigr)$. Since $\vert \cdot  \vert_\rho'$ is Lipschitz with constant $\frac{2}{\rho}$, c.f. Assumption \ref{assump: smooth approx}, we obtain

	\begin{align*}
	\lVert \dot{v}(t+h) - \dot{v}(t) \rVert_H &\leq \frac{1}{\sigma} \lVert v(t+h) - v(t) \rVert_H + \frac{1}{\sigma} \lVert f(t+h) - f(t) \rVert_V \\
	&\leq \frac{1}{\sigma} \lVert v(t+h) - v(t) \rVert_H + \frac{1}{\sigma} \lVert g(t+h) - g(t) \rVert_H + \frac{2}{\sigma \rho} \lVert \dot{z}(t+h) - \dot{z}(t) \rVert_H.
	\end{align*}
	Since $v,g,\dot z\in L^2(I,H)$, this shows $\ddot v = -\Delta \ddot z\in L^2(I,H)$.
	With the estimate
	\begin{equation*}
	\lVert \ddot{z}\rVert_{L^2(I,H^2(\Omega_0))} \leq C \big( \lVert \ddot{z} \rVert_{L^2(I,V)} + \lVert \ddot{v} \rVert_{L^2(I,H)} \big)
	\end{equation*}
the claim follows.
\end{proof}

\begin{remark} \label{rem: H2 in space and time}
This corollary and the estimates from Lemma \ref{lemma: regularity space} show that $\lVert z \rVert_{H^1(I,H^2(\Omega_0))}$ is bounded for $\rho \searrow 0$. This is not true for $\lVert \ddot{z}(t) \rVert_{H^2(\Omega_0)}$, which is not necessarily bounded for $\rho \searrow 0$.
\end{remark}

We now summarize our regularity results for the state in a theorem.

\begin{theorem} \label{theorem: regularity of the state} (Regularity of the state)

	Let $(\rho_n)_n \in \R$ be a positive and bounded sequence. Let further $(g_n)_{n\in\N}\in H^1_\star(I,V^*)$ and define $z_n := \mathcal{S}_{\rho_n}(g_n)$. Then we have the following regularity results.

	\begin{enumerate}

		\item It holds
		\begin{equation*}
		z_n \in H^1_\star(I,V)\cap H^2(I,V).
		\end{equation*}
		Furthermore, the sequence $(z_n)_n$ is bounded in these spaces if $(g_n)_n $ is bounded in $H^1_\star(I,V^*)$.

		\item If additionally $g_n\in H^1_\star(I,H)$ for all $n\in \N$, then it holds for all open and compactly contained subsets $\Omega_0$ of $\Omega$

		\begin{equation*}
		z_n \in H^1_\star(I,V)\cap H^2(I,V)\cap H^2(I,H^2(\Omega_0)).
		\end{equation*}

		Furthermore, the sequence $(z_n)_n$ is bounded in $H^1_\star(I,V), \; H^2(I,V)$ and $H^1_\star(I,H^2(\Omega_0))$ if $(g_n)_n $ is bounded in $H^1_\star(I,H)$.

	\end{enumerate}

\end{theorem}

\section{Passing to the limit in the smooth state equation} \label{section: Passing to limit in smooth equation}

In this section we analyze the regularized state equation for $\rho \searrow 0$. We will prove that in this process solutions of the smooth state equation converges to the solution of the non-smooth equation. But first we prove that the non-smooth state equation is uniquely solvable.\\

We start by proving a lemma that will give us some useful estimates. The proof uses an idea from \cite[Proof of Lemma 4.7]{SWW2016}.

\begin{lemma} \label{lemma: Cauchy property of smooth equation}

	Let $(\rho_n)_{n\in \N} \in \R $ be a sequence with $\rho_n >0$. Let further $(g_n)_{n\in \N}\in H^1_\star(I,V^*)$ be given, and define $z_n := \mathcal{S}_{\rho_n}(g_n)$. Then for all $n,m \in \N$
	we have
	\begin{equation}
	\frac{\sigma}{2} \lVert \dot{z}_n - \dot{z}_m \rVert_{L^2(I,V)}^2 + \frac{1}{2}\lVert z_n - z_m \rVert_{C(\bar{I},V)}^2 \leq 2  T \vert \Omega \vert \cdot \vert \rho_m - \rho_n \vert + \frac{2}{\sigma} \lVert g_n - g_m \rVert_{L^2(I,V^*)}. \label{eq: state Cauchy sequence}
	\end{equation}

\end{lemma}

\begin{proof}

	We test the equations
	\begin{align*}
	-\sigma \Delta\dot{z}_n - \Delta z_n - g_n + \vert \dot{z}_n \vert_{\rho_n}' &= 0, \\
	-\sigma \Delta\dot{z}_m - \Delta z_m - g_m + \vert \dot{z}_m \vert_{\rho_m}' &= 0
	\end{align*}
	with $\dot{z}_n - \dot{z}_m$, subtract them from each other, and integrate from $0$ to $t$. This yields
	\begin{align} \label{proof: smooth Cauchy1}
	\sigma \lVert  &\dot{z}_n - \dot{z}_m \rVert_{L^2(0,t;V)}^2 + \frac{1}{2} \lVert z_n(t) - z_m(t) \rVert_{V}^2 - \frac{1}{2} \lVert \underbrace{z_n(0)}_{=0} - \underbrace{z_m(0)}_{=0}\rVert_{V}^2 \notag \\
	&= - \int\limits_{0}^{t}  \int\limits_{\Omega} \Bigl( \vert \dot{z}_n \vert_{\rho_n}' - \vert \dot{z}_m \vert_{\rho_m}' \Bigr)\cdot \Bigl( \dot{z}_n - \dot{z}_m \Bigr) \dx \,\ds + \int\limits_{0}^{t} \langle g_n - g_m, \dot{z}_n -\dot{z_m} \rangle \ds.
	\end{align}
	The convexity of $\vert \cdot \vert_\rho$ and Property \ref{assumption, estimate two rhos} from Assumption \ref{assump: smooth approx} imply
	\begin{multline} \label{proof: smooth Cauchy2}
	- \int\limits_{0}^{t} \int\limits_{\Omega} \Bigl( \vert \dot{z}_n \vert_{\rho_n}' - \vert \dot{z}_m \vert_{\rho_m}' \Bigr)\cdot \Bigl( \dot{z}_n - \dot{z}_m \Bigr) \dx\ds\notag\\
	\begin{aligned}
	&
\leq \int\limits_{0}^{t}  \int\limits_{\Omega} \vert \dot{z}_n \vert_{\rho_m} - \vert \dot{z}_m \vert_{\rho_m}  \dx\,\ds
	+ \int\limits_{0}^{t} \int\limits_{\Omega} \vert \dot{z}_m \vert_{\rho_n} - \vert \dot{z}_n \vert_{\rho_n}  \dx\,\ds \notag \\
	&= \int\limits_{0}^{t}  \int\limits_{\Omega} \underbrace{\Big\vert  \vert \dot{z}_m \vert_{\rho_n} - \vert \dot{z}_m \vert_{\rho_m} \Big\vert}_{\leq  \vert \rho_n - \rho_m \vert}  \dx\,\ds
	+ \int\limits_{0}^{t} \int\limits_{\Omega} \underbrace{\Big\vert  \vert \dot{z}_n \vert_{\rho_m} - \vert \dot{z}_n \vert_{\rho_n} \Big\vert}_{\leq  \vert \rho_n - \rho_m \vert} \dx\,\ds \leq 2T \, \vert \Omega \vert \cdot \vert \rho_n - \rho_m \vert.
	\end{aligned}
	\end{multline}
	Furthermore, using  Hölder's and Young's inequality gives
	\begin{equation*}
	\int\limits_{0}^{t} \langle g_n - g_m, \dot{z}_n -\dot{z}_m \rangle_{V^*,V}
\leq \frac{1}{2\sigma} \lVert g_n - g_m \rVert_{L^2(I;V^*)}^2 + \frac{\sigma}{2}\lVert \dot{z}_n - \dot{z}_m \rVert_{L^2(I,V)}^2.
	\end{equation*}
	Applying the previous estimates in equation \eqref{proof: smooth Cauchy1}, yields
	\begin{equation*}
	\frac{\sigma}{2} \lVert  \dot{z}_n - \dot{z}_m \rVert_{L^2(0,t;V)}^2 + \frac{1}{2} \lVert z_n(t) - z_m(t) \rVert_{V}^2
	\leq 2  T \vert \Omega \vert \cdot \vert \rho_m - \rho_n \vert +\frac{1}{2\sigma} \lVert g_n - g_m \rVert_{L^2(I;V^*)}^2.
	\end{equation*}
	f.a.a. $t\in I$, which is the asserted inequality.
\end{proof}

We are now ready to prove existence and uniqueness of solutions for the non-smooth state equation.

\begin{proof}[Proof of Theorem \ref{theorem:  solution non-smooth}] \label{Proof: ex and unique of non-smooth}
	Let $g\in H^1_\star(I,V^*)$ be given. Let us take a sequence $(\rho_n)_{n\in \N} \in\R$ be a sequence with $\rho_n \searrow 0$
	and define $z_n:= \mathcal{S}_{\rho_n} g $. Using Theorem \ref{theorem: regularity of the state} we obtain that the sequence $(z_n)_n$ is bounded in $H^1_\star(I,V) \cap H^2(I,V) $. Due to reflexivity, we have a weakly convergent subsequence (which we denote again by $z_n$ ) and a function $ z\in H^2(I,V) \cap H^1_\star(I,V)$ such that $z_n \rightharpoonup z$ in these spaces.

	Moreover, we have $\lVert \dot{z_n}(0)\rVert_{V} \leq \frac{\rho_n}{\sigma} \vert \Omega \vert$, see \eqref{eq: estimate dotz_0}, and hence $\dot{z}(0)= 0$ is satisfied.

	Lemma \ref{lemma: Cauchy property of smooth equation} shows, that $z_n$ is a Cauchy sequence in $H^1_\star(I,V)$, which implies $z_n \to z$ in $H^1_\star(I,V)$. Due to the convexity of $\vert \cdot \vert_\rho$ we have for a.a. $t\in I$, all $n\in \N$ and all $v\in V$

	\begin{equation} \label{proof: existence uniqueness of nonsmooth1}
	\int\limits_{\Omega} \vert v \vert_{\rho_n} \dx \geq \int\limits_{\Omega} \vert \dot{z}_n(t) \vert_{\rho_n} \dx
	+ \langle \sigma \Delta \dot{z}_n(t) + \Delta z_n(t) + g(t),v-\dot{z}_n \rangle_{V^*,V}.
	\end{equation}
	It is easy to show that we can pass to the limit in this inequality and obtain
	\begin{equation*}
	\lVert v \rVert_{L^1(\Omega)} \geq \lVert \dot{z}(t)\rVert_{L^1(\Omega)} + \langle \sigma \Delta \dot{z}(t) + \Delta z(t) + g(t),v-\dot{z}(t) \rangle_{V^*,V},
	\end{equation*}
	i.e. $\sigma \Delta \dot{z}(t) + \Delta z(t) + g(t) \in \partial \vert \dot{z}(t)\vert$.\\

	It remains to prove uniqueness of solutions.
	Let two solutions $z_1,z_2 \in H^1_\star(I,V)$ of the non-smooth state equation be given. Then for all $v,w\in V$ and a.a. $t\in I$ we have
	\begin{align*}
	\lVert v \rVert_{L^1(\Omega)} &\geq \lVert \dot{z}_1(t) \rVert_{L^1(\Omega)} + \langle \sigma \Delta \dot{z}_1(t) + \Delta z_2(t) + g(t), v- \dot{z}_1(t) \rangle_{V^*,V}, \\
	\lVert w \rVert_{L^1(\Omega)} &\geq \lVert \dot{z}_2(t) \rVert_{L^1(\Omega)} + \langle \sigma \Delta \dot{z}_2(t) + \Delta z_2(t) + g(t), w- \dot{z}_2(t) \rangle_{V^*,V}.
	\end{align*}
Choosing $v:= \dot{z}_2(t)$, $w:= \dot{z}_1(t)$, adding the resulting inequalities, and canceling out some summands gives
	\begin{equation*}
	0 \geq \lVert \dot{z}_1(t) - \dot{z}_2(t) \rVert_{V}^2 + \bigl( z_2(t) - z_1(t), \dot{z}_2(t) - \dot{z}_1(t) \bigr)_{V}.
	\end{equation*}
	Integrating this inequality from $0$ to $t$ yields
	\begin{equation*}
	0 \geq \lVert \dot{z}_1(t) - \dot{z}_2(t) \rVert_{L^2(0,t;V}^2 + \frac{1}{2} \lVert z_2(t) - z_1(t) \rVert_{V}^2 - \frac{1}{2} \lVert \underbrace{z_2(0)}_{=0} - \underbrace{z_1(0)}_{=0} \rVert_{V}^2,
	\end{equation*}
	hence $z_1 = z_2$ on $I$.
\end{proof}

In particular, this proof yields the following corollary.

\begin{corollary} \label{cor: convergence solutions}
	For every $g\in H^1_\star(I,V^*)$ we have $\mathcal{S}_{\rho}(g)\to \mathcal{S}(g)$ in $H^1_\star(I,V)$ for $\rho \searrow 0$.
\end{corollary}

In the next theorem we show stronger convergence for $\rho \searrow 0$.

\begin{theorem} \label{theorem: GW glatt}
	Let sequences  $(\rho_n)_{n\in \N}\in \R$ and $(g_n)_{n\in N} \in H^1_\star(I,V^*)$ be given with $\rho_n \searrow0$
	 and $g_n \to g$ in $L^2(I,V^*)$ for some $g \in H^1_\star(I,V^*)$. Define $z_n := \mathcal{S}_{\rho_n}(g_n)$ and $z := \mathcal{S}(g)$. Then it holds $z_n \to z$ in $H^1_\star(I,V)$ and $\mathcal{C}(\bar{I},V)$.
\end{theorem}
\begin{proof}
We have the estimate
\begin{equation*}
\lVert z_n  - z \rVert_{H^1_\star(I,V)} \leq  \lVert z_n  - \mathcal{S}_{\rho_n}(g) \rVert_{H^1_\star(I,V)} + \lVert \mathcal{S}_{\rho_n}(g)   - z \rVert_{H^1_\star(I,V)}.
\end{equation*}
The right-hand side converges to zero due to Lemma \ref{lemma: Lipschitz smooth solution operator} and Corollary \ref{cor: convergence solutions}. The claim follows by passing to the limit $m\to\infty$ in inequality \eqref{eq: state Cauchy sequence}.
\end{proof}

\section{Optimality system} \label{chapter: Optimality system}

So far we studied the smooth state equation and the behavior of solutions for $\rho \searrow 0$. Now we would like to find an optimality system for \eqref{eq: non-smooth problem}. We will formulate such a system for an optimal control problem with the regularized state equation and then pass to the limit $\rho \searrow 0$.
However, first we need to know how the optimality system should look like. This will be discussed in the following subsection.

\subsection{Formal derivation of an optimality system} \label{chapter: Formal derivation optimality system}

Motivated by the previous thoughts we formally derive optimality conditions for the non-smooth optimal control problem.
Consider the optimal control problem
\begin{align*}
&\min J(z,g) \\
&\text{s.t.} \; \bigl( \dot{z}(t,x),g(t,x) + \Delta z(t,x) + \sigma \Delta\dot{z}(t,x)  \bigr) \in M \qquad \forall (t,x) \in I\times \Omega,
\end{align*}
where
\begin{equation*}
M :=\text{gph}\,\partial \vert \cdot \vert = \{(u,v) \in \R^2 \mid v \in \partial \vert u \vert \}
= \bigl( (-\infty,0] \times \{-1\} \bigr)  \cup \bigl( \{0\} \times [-1,1] \bigr) \cup \bigl( [0,\infty) \times \{1\} \bigr).
\end{equation*}
In \cite[Chapter 2]{SWW2016} optimality conditions are formally derived by using the Lagrangian
\begin{equation*}
\mathcal{L}(z,g,q,\xi) := J(z,g) - (q,\dot{z})_{L^2(I\times \Omega)} + (\xi,g+\Delta z)_{L^2(I\times \Omega)}.
\end{equation*}
We present another way to derive optimality conditions, which however gives the same conditions as the approach from \cite{SWW2016}.
Using the indicator function of the set $M$ it is possible to write the optimal control problem as an unconstrained problem,
\begin{equation*}
\min J(z,g) + \delta_M( \dot{z},g + \Delta z + \sigma \Delta\dot{z}) \qquad (z,g) \in H^1_\star(I,V)\times H^1_\star(I,H).
\end{equation*}
Using the generalized Fermat rule, an optimality condition is given by
\begin{align*} \label{Fermat}
0 &\in \partial \bigl[J(z,g) + \delta_M( \dot{z},g+ \Delta z + \sigma \Delta\dot{z}) \bigr]\\
 &= j_1'(z) +\langle j_2'(z(T)), \hat{\delta}_T\rangle_{V^*,V} + \bigl( g,\cdot \bigr)_{H^1(I,H)} + \partial \delta_M( \dot{z},g+ \Delta z + \sigma \Delta\dot{z}),
\end{align*}
where $\hat\delta_T:w\mapsto w(T)$ denotes the evaluation of a function at time $T$.

Let us define
$(z,g) \mapsto L(z,g) := \bigl( \dot{z},g + \Delta z + \sigma \Delta\dot{z} \bigr) $, which
is a linear mapping between Hilbert spaces with (formal) adjoint
\begin{equation*}
L^*(q,\xi):= \Bigl( -\dot{q} + \Delta \xi - \sigma \Delta \dot{\xi} + (q(T),\hat{\delta}_T)_H + (\Delta \xi(T),\hat{\delta}_T)_H,\; \xi \Bigr).
\end{equation*}
Continuing our formal calculations, we apply the chain rule in the form
 $\partial (\delta_M\circ L) = L^*\circ \partial \delta_M  \circ L$.
Then we arrive at
\begin{equation*}
\partial  \delta_M( \dot{z},g+ \Delta z + \sigma \Delta\dot{z}) = \bigl( L^* \circ N_M \bigr)(\dot{z},g+ \Delta z + \sigma \Delta \dot{z}),
\end{equation*}
where $N_M$ is the Fr\'echet normal cone of $M$.
This implies that there is $(-q,\xi)\in N_M(\dot{z},g+ \Delta z + \sigma \Delta \dot{z})$
such that
\begin{align*}
0 \in \Bigl( j_1'(z) + &\langle j_2'(z(T)), \hat{\delta}_T \rangle_{V^*,V} +\dot{q} + \Delta \xi - \sigma \Delta \dot{\xi} + (-q(T),\hat{\delta}_T)_H + \sigma (\Delta \xi(T),\hat{\delta}_T)_H,\\
 &  -\ddot{g}+ g + (\dot{g}(T),\hat{\delta}_T)_H + \xi \Bigr).
\end{align*}
Hence, (formal!) optimality conditions are given by
\begin{subequations} \label{formal optimality system equations}
	\begin{align}
	j_1'(z) + \dot{q} +\Delta \xi - \sigma \Delta \dot{\xi} &= 0 &&\text{a.e.\@ on } I\times \Omega,\label{eq51a} \\
	j_2'(z(T)) - q(T) + \sigma \Delta \xi(T)&=0 &&\text{a.e.\@ on } \Omega, \\
	g-\ddot{g}+\xi &=0 &&\text{a.e.\@ on } I\times \Omega, \\
	\dot{g}(T)&=0 &&\text{a.e.\@ on } \Omega,\label{eq51d}\\
	(-q,\xi)&\in N_M(\dot{z},g+ \Delta z + \sigma \Delta \dot{z}).\label{eq51e}
	\end{align}
\end{subequations}
The condition \eqref{eq51e} involving the Fr\'{e}chet normal cone of $M$
can be written as the following system of pointwise properties:
\begin{subequations} \label{formal optimality system cone}
	\begin{align}
	\dot{z}(t,x)>0,\quad g(t,x) + \Delta z(t,x) + \sigma \Delta \dot{z}(t,x) = 1\quad &\Longrightarrow \quad q(t,x)= 0, \label{formal optimality system cone a} \\
	\dot{z}(t,x)=0,\quad g(t,x) + \Delta z(t,x) + \sigma \Delta \dot{z}(t,x) = 1\quad &\Longrightarrow \quad q(t,x)\geq 0,\quad \xi(t,x) \geq 0, \label{formal optimality system cone b} \\
	\dot{z}(t,x)=0,\quad \vert g(t,x) + \Delta z(t,x) + \sigma \Delta \dot{z}(t,x)\vert < 1 \quad &\Longrightarrow \quad \xi(t,x)= 0, \label{formal optimality system cone c} \\
	\dot{z}(t,x)=0,\quad g(t,x) + \Delta z(t,x) + \sigma \Delta \dot{z}(t,x) = -1\quad &\Longrightarrow \quad q(t,x)\leq 0,\quad \xi(t,x)\leq 0, \label{formal optimality system cone d} \\
	\dot{z}(t,x)<0,\quad g(t,x) + \Delta z(t,x) + \sigma \Delta \dot{z}(t,x) = -1 \quad &\Longrightarrow \quad q(t,x)= 0. \label{formal optimality system cone e}
	\end{align}
\end{subequations}

Our aim in the next sections is to prove that some of the these formally derived conditions are optimality conditions for \eqref{eq: non-smooth problem}.

\subsection{Smooth optimal control problem}
Now, we are going to consider an optimal control problem with  regularized state equation depending on the parameter $\rho$.
We show existence of solutions for this problem and investigate what happens with them for $\rho \searrow 0$. This section is based on \cite[Section 4.4]{SWW2016}.
\\
Let
\begin{equation*}
(\bar{z},\bar{g}) \in H^1_\star(I,V) \times H^1_\star(I,H)
\end{equation*}
be a local solution of \eqref{eq: non-smooth problem}. Hence, there exists $\delta >0$ such that $J(\bar{z},\bar{g})\leq J(\mathcal{S}(g),g)$ for all $g\in H^1_\star(I,H)$ with  $\lVert g-\bar{g} \rVert_{H^1(I,H)} < \delta $
holds.
We define
\begin{align}
J_\rho:  H^1(I,V) \times H^1_\star(I,H) \; &\to \; \R  \notag \\
(z,g) &\mapsto J(z,g) + \frac{1}{2} \lVert g-\bar{g} \rVert_{H^1(I,H)}^2.
\end{align}
We will now consider the optimal control problem
\begin{align}
&\min J(z,g) + \frac{1}{2} \lVert g-\bar{g} \rVert_{H^1(I,H)}^2 \quad \text{for}\;
(z,g) \in H^1(I,V) \times H^1(I,H), \notag  \\
&\text{s.t.} \left\{\!\begin{aligned}& \lVert g-\bar{g} \rVert_{H^1(I,H)} \leq \delta, \\
&g(0) = 0, \; z(0) = 0, \\
&-\sigma \Delta \dot{z} -\Delta z -g +\vert \dot{z} \vert''_\rho  = 0  \quad \text{in } V^*\quad \text{for a.a.}\; t\in I.
\end{aligned}\right.
\tag{$P_\rho$} \label{eq: smooth optimal control problem}
\end{align}
Augmenting the original problem with additional penalty terms and constraints is a well-known technique for
nonsmooth optimal control problem.
Here, the constraint $\lVert g-\bar{g} \rVert_{H^1(I,H)} \leq \delta$ will give us the existence of global solutions
of such a smoothed problem.
The additional term $\frac{1}{2} \lVert g-\bar{g} \rVert_{H^1(I,H)}^2$ will be used to force strong convergence of solutions of \eqref{eq: smooth optimal control problem} to $\bar{g}$ for $\rho \searrow 0$, see the proof of Theorem \ref{theorem: convergence solutions}, which is from \cite[Proof of Theorem 4.9]{SWW2016}.

\begin{lemma}
For all $\rho>0$
	the optimal control problem \eqref{eq: smooth optimal control problem} has global solutions $(z_\rho,g_\rho)$.
	In addition it holds $z_\rho \in H^2_\star(I,V) $.
\end{lemma}

\begin{proof}
	The proof of existence is similar to the proof of Lemma \ref{lemma: global solution nonsmooth}. The optimal state $z_\rho$ is in the asserted spaces due to Theorem \ref{theorem: regularity of the state}.
\end{proof}

For convergence of global solutions, we have the following theorem.
Its proof is identical to the proof of \cite[Theorem 4.9]{SWW2016}.

\begin{theorem} \label{theorem: convergence solutions}
	Let $\{(z_\rho,g_\rho)\}_{\rho >0}$ a family of global solutions of \eqref{eq: smooth optimal control problem}. Then for $\rho \searrow 0$ it holds
	\begin{equation*}
	\begin{array}{lll}
	g_\rho \to \bar{g}  &\text{in} & H^1_\star(I,H) ,\\
	z_\rho \to \bar{z}  &\text{in} & H^1_\star(I,V).
	\end{array}
	\end{equation*}
\end{theorem}

\subsection{Optimality system for the regularized problem}
In this section we will formulate optimality conditions for the regularized problem \eqref{eq: smooth optimal control problem}.
Motivated by the results of Section \ref{chapter: Formal derivation optimality system} the optimality system should include the equations
\begin{align*}
-\dot{q}_\rho  + \sigma \Delta \dot{\xi}_\rho &= \Delta \xi + j_1'(z_\rho),  \\
q_\rho(T) - \sigma \Delta \xi_\rho(T) &= j_2'(z_\rho(T)).
\end{align*}
Moreover, we would like to have an optimality condition corresponding to the inclusion $(-q,\xi)\in N_M(\dot{z},g+ \Delta z + \sigma \Delta \dot{z})$, see \eqref{formal optimality system equations}. Let us formulate a version of this condition for the regularized problem. Define $M_\rho := \operatorname{graph}(\vert \cdot \vert_\rho')\subset \R^2$ and consider

\begin{equation*}
(-q_\rho,\xi_\rho)\in N_{M_\rho}\big(\dot{z}_\rho,g_\rho+ \Delta z_\rho + \sigma \Delta \dot{z}_\rho \big) = N_{M_\rho}\big( \dot{z}_\rho, \vert \dot{z}_\rho \vert_\rho' \big),
\end{equation*}
which is equivalent to the equation
\begin{equation*}
0 = \Big\langle  \begin{pmatrix}
-q_\rho\\
\xi_\rho
\end{pmatrix}
,
\begin{pmatrix}
1\\
\vert \dot{z}_\rho \vert_\rho''
\end{pmatrix}
  \Big\rangle_{\R^2}
  = -q_\rho + \xi_\rho \vert \dot{z}_\rho \vert_\rho''.
\end{equation*}
We introduce the substitution $u_\rho := q_\rho - \sigma \Delta \xi_\rho$ and define the optimality system
\begin{subequations} \label{NOB glatt}
	\begin{align}
	j_1'(z_\rho) + \dot{u}_\rho +\Delta \xi_\rho &= 0 \qquad &\text{in } V^*\; &\text{a.e.\@ on } I, \label{NOB glatt a} \\
	j_2'(z_\rho(T)) - u_\rho(T)&=0 &\text{in } V^*, \label{NOB glatt b}\\
	-\sigma \Delta \xi_\rho + \vert \dot{z_\rho} \vert_\rho'' \xi_\rho &= u_\rho &\text{in } V^*\; &\text{a.e.\@ on } I, \label{NOB glatt c}
	\end{align}
\end{subequations}
which is similar to the system in  \cite{SWW2016} but  with $q_\rho$ replaced by $u_\rho$.
We call $\xi_\rho,u_\rho,q_\rho$ \emph{adjoint variables} and system \eqref{NOB glatt} the \emph{adjoint system}.
First we show that the adjoint system is uniquely solvable (in certain function spaces).

\begin{lemma}
	Let $\rho >0$ and $(z_\rho, g_\rho) \in H^1_\star(I,V)\times H^1_\star(I,H)$ with $z_\rho = \mathcal{S}_\rho(g_\rho)$ be given. Then there exists a unique solution $\bigl( u_\rho,\xi_\rho \bigr)\in H^1(I,V^*)\times L^2(I,V)$ of system \eqref{NOB glatt}. If $d\leq 4$ then $\xi_\rho \in H^1(I,V)$. In this case there exists a constant $C>0$ independent of $\rho,\xi_\rho,u_\rho,z_\rho$ such that

	\begin{equation*}
	\lVert \dot{\xi}_\rho \rVert_{L^2(I,V)} \leq  \frac{2}{\rho^2} \lVert \ddot{z}_\rho \rVert_{L^\infty(I,L^2(\Omega))} \cdot
	\lVert \xi_\rho \rVert_{L^2(I,V)} + \lVert \dot{u}_\rho \rVert_{L^2(I,V^*)} .
	\end{equation*}
\end{lemma}

\begin{proof}
	Existence and uniqueness can be shown in the same way as in \cite[Proof of Lemma 4.13]{SWW2016}.  Observe that the adjoint system is equivalent to
	\begin{align}
	\dot{u}_\rho(t) &= - \Delta T_\rho'\big(g_\rho(t) + \Delta z_\rho(t)\big)u_\rho(t) - j_1'(z_\rho) &&\text{in } V^* \; \text{a.e.\@ on } I, \label{NOB glatt equiv a}\\
	u_\rho(T) &= j_2'(z_\rho(T)) &&\text{in } V^*, \label{NOB glatt equiv b} \\
	\xi_\rho(t) &= T_\rho'(g_\rho(t) + \Delta z_\rho(t))u_\rho(t) &&\text{in } V \; \text{a.e.\@ on } I. \label{NOB glatt equiv c}
	\end{align}

	The first two equations are an initial value problem in the Banach space $V^*$, which is uniquely solvable, see \cite[Satz 1.3]{GGZ1974}. Hence, the adjoint system is uniquely solvable.
	It remains to prove $\xi_\rho \in H^1(I,V)$ in the case $d\leq 4$. That is, we have to show differentiability of $\xi_\rho$ in time.
	Here, we investigate the differences $\xi_\rho(t+h) - \xi_\rho(t)$. By definition of $\xi_\rho$ we have f.a.a.\@ $t\in I$
	\begin{align*}
	- \sigma \Delta \xi_\rho(t) + \vert \dot{z}_\rho(t) \vert_\rho'' \xi_\rho(t) &= u_\rho(t) \quad &\text{in } V^*, \\
	- \sigma \Delta \xi_\rho(t+h) + \vert \dot{z}_\rho(t+h) \vert_\rho'' \xi_\rho(t+h) &= u_\rho(t+h) \quad &\text{in } V^*.
	\end{align*}

	We subtract these equations from each other, add, and subtract the term $\vert \dot{z}_\rho(t+h)\vert_\rho'' \xi_\rho(t)$ to obtain

	\begin{equation*}
	-\sigma \Delta \bigl(  \xi_\rho(t+h)-\xi_\rho(t) \bigr) + \vert \dot{z}_\rho(t+h)\vert_\rho''  \bigl(  \xi_\rho(t+h)-\xi_\rho(t) \bigr)
	= - \bigl( \vert \dot{z}_\rho(t+h)\vert_\rho''- \dot{z}_\rho(t)\vert_\rho'' \bigr) \xi_\rho(t) + u_\rho(t+h) - u_\rho(t)
	\end{equation*}
By the Lax-Milgram theorem, we get the estimate
	\begin{equation*}
	\big\lVert \xi_\rho(t+h) - \xi_\rho(t) \big\rVert_{V}
	\leq \frac{1}{\sigma}  \big\lVert \bigl( \vert \dot{z}_\rho(t+h)\vert_\rho''- \vert \dot{z}_\rho(t)\vert_\rho'' \bigr) \xi_\rho(t)\big\rVert_{V^*} + \big\lVert u_\rho(t+h) - u_\rho(t) \big\rVert_{V^*}.
	\end{equation*}
	The Lipschitz continuity of  $\vert \cdot \vert_\rho''$, c.f. Assumption \ref{assump: smooth approx}, implies for almost all $t\in I$
	\begin{align*}
	\big\lVert \xi_\rho(t+h) - \xi_\rho(t) \big\rVert_{V}  \leq  \frac{2}{\rho^2} \underbrace{\big\lVert \bigl(  \dot{z}_\rho(t+h)-  \dot{z}_\rho(t) \bigr) \xi_\rho(t) \big\rVert_{L^\frac{4}{3}(\Omega)}}_{=: A_h} + \big\lVert u_\rho(t+h) - u_\rho(t) \big\rVert_{V^*}.
	\end{align*}
	Note that the term $A_h$ is well defined, since $d\leq 4$ implies the embedding $V\hookrightarrow L^4(\Omega)$. Using Hölder's inequality we obtain
	\begin{equation}
	\lVert \xi_\rho(t+h) - \xi_\rho(t) \rVert_{H^1_0(\Omega)}
	\leq \big\lVert \dot{z}_\rho(t+h) - \dot{z}(t) \big\rVert_{L^2(\Omega)} \cdot \ \big\lVert \xi_\rho(t) \big\rVert_{L^4(\Omega)}  + \lVert u_\rho(t+h) - u_\rho(t) \rVert_{V^*}. \label{proof: xi diffbar}
	\end{equation}

%
%

	Since $u_\rho \in H^1(I,V^*)$ and $\dot{z}_\rho \in H^1(I,V) $ the finite differences $\frac{1}{h}\bigl( \xi_\rho(t+h)-\xi_\rho(t)\bigr)$ are bounded in $V$, hence $\xi$ is differentiable a.e.\@ in $I$. We divide inequality \eqref{proof: xi diffbar} by $h$ and obtain the asserted inequality by passing to the limit $h \searrow 0$.
Squaring this inequality and integrating over $I$ proves $\dot{\xi}_\rho \in L^2(I,V)$. Observe that $\xi_\rho \in L^\infty(I,V)$, since $u \in H^1(I,V^*)\hookrightarrow L^\infty(I,V^*)$ and $\lVert \xi_\rho(t)\rVert_V \leq \frac{1}{\sigma} \lVert u(t)\rVert_{V^*}$.
\end{proof}

Given the unique solvability of the adjoint system, we can formulate the optimality conditions for the regularized optimal control problem.

\begin{theorem} \label{theorem: optimality conditions smooth}
	Let $\rho >0$ and $(z_\rho,g_\rho)$ be a local solution of \eqref{eq: smooth optimal control problem} with $\lVert g_\rho - \bar{g} \rVert_{H^1(I,H)} < \delta$. Then there exist unique $\bigl( u_\rho,\xi_\rho \bigr)\in H^1(I,V^*)\times L^2(I,V)$ which is the solution of \eqref{NOB glatt} and
	\begin{align*}
	-2 \ddot{g}_\rho + \ddot{\bar{g}} + 2 g_\rho - \bar{g} + \xi_\rho &= 0 \quad &\text{for a.a.}\;t\in I \; \text{in } H^{-1}(\Omega),\\
	g_\rho(0)&=0 &\text{a.e. in}\; V^*,\\
	2\dot{g}_\rho(T)-\dot{\bar{g}}(T) &= 0 &\text{a.e. in}\; V^*.
	\end{align*}
	holds in the following weak sense: \\

	\begin{equation*}
	\bigl( h,\xi_\rho \bigr)_{L^2(I,H)} + \bigl( 2 g_\rho - \bar{g}, h \bigr)_{H^1(I,H)} = 0 \qquad \forall \; h\in H^1_\star(I,H)
	\end{equation*}
\end{theorem}

\begin{proof}
The proof is exactly as \cite[Proof of Theorem 4.14]{SWW2016}.
\end{proof}

\begin{remark} \label{remark j1}
	Recall that $j_1$ is a mapping from $L^2(I,V)$ to $\R$, which means it holds

	\begin{equation*}
	j_1'(z_\rho) \in \bigl[L^2(I,V)\bigr]^* \cong L^2(I,V^*),
	\end{equation*}
In the following, we will always denote by  $j_1'(z_\rho) $ the representative in $L^2(I,V^*)$.
\end{remark}

We would like to prove that the optimality conditions converge for $\rho \searrow 0$ to the equations that we derived in Section \ref{chapter: Formal derivation optimality system}.
To this end we have to show that $\xi_\rho, u_\rho$ converge (weakly) in suitable function spaces. Hence our next aim is to prove some boundedness properties.

\begin{lemma}
Let $g_\rho \in H^1(I,H)$ and $z_\rho := \mathcal{S}_\rho(g_\rho) $. Let further $(u_\rho,\xi_\rho)$ be the unique solution of \eqref{NOB glatt}. Then we have f.a.a. $t\in I$ the estimates

\begin{subequations}
	\begin{align}
		\lVert u_\rho (t)\rVert_{V^*} &\leq e^{\frac{1}{\sigma }T} \Big[ \lVert j_1'(z_\rho)  \rVert_{L^1(I,V^*)} + \lVert j_2'(z_\rho(T)) \rVert_{V^*}  \Big], \\
		\lVert \dot{u}_\rho(t) \rVert_{V^*}  &\leq \big\lVert \big[ j_1'(z_\rho)\big](t) \rVert_{V^*} + \lVert \xi_\rho(t) \big\rVert_V , \\
		\frac{\sigma}{2} \lVert \xi_\rho(t) \rVert_V^2 + \intom \vert \dot{z}_\rho(t) \vert_\rho'' \xi_\rho(t)^2 \, \dx
		&\leq e^{\frac{1}{\sigma}T} \frac{1}{2 \sigma} \Big[ \lVert j_1'(z_\rho)  \rVert_{L^2(I,V^*)} + \lVert j_2'(z_\rho(T)) \rVert_{V^*}  \Big].
	\end{align}
\end{subequations}

\end{lemma}

\begin{proof}
	Integrating \eqref{NOB glatt equiv a} from $t$ to $T$ we obtain

	\begin{equation*}
	u_\rho(T) - u_\rho(t) = - \Delta \int\limits_t^T T_\rho'(g_\rho(s)+ \Delta z_\rho(s) ) u_\rho(s) - \big[j_1'(z_\rho)\big](s) \, \ds
	.
 	\end{equation*}
	Using \eqref{NOB glatt equiv b} we obtain the estimate

	\begin{equation*}
	\lVert u_\rho(t) \rVert_{V^*} \leq \lVert j_2'(z_\rho(T)) \rVert_{V^*}  + \int\limits_t^T \lVert j_1'(z_\rho) \rVert_{L^1(I,V^*)} + \frac{1}{\sigma} \int\limits_t^T \lVert u_\rho(s) \rVert_{V^*},
	\end{equation*}
	and Gronwall's inequality yields the first inequality. The second one follows immediately from \eqref{NOB glatt a}.
	In order to prove the third estimate we test \eqref{NOB glatt a} with $\xi_\rho(t)$ and use Young's inequality:
	\begin{align*}
	\sigma \big\lVert \xi_\rho(t)\big\lVert_V^2 + \int\limits_\Omega\xi_\rho(t)^2 \vert \dot{z}_\rho(t) \vert_\rho'' \; \dx
	&= \langle u_\rho(t) ,\xi_\rho(t) \rangle_{V^*,V} \leq \frac{\sigma}{2}  \big\lVert  \xi_\rho  \big\rVert_V^2 + \frac{1}{2\sigma}  \big\lVert u_\rho(t)  \big\rVert_{V^*}^2,
	\end{align*}
which finishes the proof.
\end{proof}

This lemma gives us some boundedness properties, which we will collect next.
Recall that we defined in the beginning of the section $q_\rho = u_\rho + \sigma \Delta \xi_\rho$.

\begin{corollary} \label{cor: Beschränktheit Adjungiert}
	Let the family $\{z_\rho,g_\rho\}_{\rho >0}$ with $z_\rho := \mathcal{S}_\rho(g_\rho)$ be bounded in $H^1_\star(I,V)\times H^1_\star(I,H)$. Then there exists a constant $C>0$ independent of $\rho,g_\rho, z_\rho,\xi_\rho, u_\rho$ such that
	\begin{align*}
	&\lVert j_1'(z_\rho)\rVert_{L^2(I,V^*)}, \quad
	\lVert j_2'(z_\rho(T)) \rVert_{V},\\
	&\lVert \xi_\rho  \rVert_{L^\infty(I,V)},
	\quad \lVert u_\rho  \rVert_{W^{1,\infty}(I,V^*)}, \quad \lVert q_\rho  \rVert_{L^\infty(I,V^*)}, \\
	&\Big\lVert  \vert  \dot{z}_\rho \vert_\rho'' \xi_\rho^2 \Big\rVert_{L^\infty(I,L^1(\Omega))}
	\end{align*}
	are less then $C$.
\end{corollary}

\subsection{Optimality system for the non-smooth problem}
\label{sec74}

In this section we analyze the optimality system for $\rho \searrow 0$. 
We start with a lemma that shows a weak formulation of the optimality conditions \eqref{formal optimality system cone a}, \eqref{formal optimality system cone e} and corresponds to \cite[Lemma 5.1]{SWW2016}.

\begin{lemma} \label{lemma: optimality condition1}
	Let $(z_\rho )_{\rho>0} \in H^1_\star(I,V) $ with $z_\rho \to z$ in $H^1_\star(I,V)$. Let further $(u_\rho,\xi_\rho )_{\rho>0}$ be the corresponding unique solutions of the adjoint system \eqref{NOB glatt} and define $q_\rho := u_\rho + \sigma \Delta \xi_\rho$. Then there exists a function $q\in L^\infty(I,V^*)$
	and a subsequence of $q_\rho$, which we denote again by $q_\rho$, such that $q_\rho \rightharpoonup^* q$ in $L^\infty(I,V^*)$ for $\rho \searrow 0$. Furthermore, for a.a. $t\in I$ we have

	\begin{equation*}
	\langle q(t),\phi \vert \dot{z}(t)\vert \rangle_{V^*,V} = 0 \qquad \forall \phi \in C_0^\infty(\Omega).
	\end{equation*}

\end{lemma}

\begin{proof}
	Let $\phi \in C_0^\infty(\Omega)$ and $\eta \in C_0^\infty(I)$ be given.
	Testing the adjoint equation $q_\rho(t) = \vert \dot{z}_\rho(t) \vert_\rho''\xi_\rho(t)$ with $\phi \;\vert \dot{z}_\rho(t) \vert \; \eta(t) $ we obtain for almost all $t\in I$
	\begin{equation*}
	\big\vert \langle q_\rho(t),\phi \, \vert \dot{z}_\rho(t)\vert\, \eta(t) \rangle_{V^*,V} \big\vert
 \leq \lVert \phi \rVert_{L^\infty(\Omega)} \cdot \lVert \eta \rVert_{L^\infty(I)} \cdot \Big\lVert \sqrt{\vert \dot{z}_\rho(t) \vert_\rho''} \; \dot{z}_\rho(t)  \Big\rVert_{H} \cdot \Big\lVert \sqrt{\vert \dot{z}_\rho(t) \vert_\rho''} \;  \xi_\rho(t)  \Big\rVert_{H}.
	\end{equation*}
	By Lemma \ref{lemma: properties approx} we have $\sqrt{\vert \dot{z}_\rho(t,x) \vert_\rho''} \; \vert \dot{z}_\rho(t,x) \vert \leq \sqrt{2\rho}$. Furthermore, we know due to Corollary \ref{cor: Beschränktheit Adjungiert} that
	\begin{equation*}
\Big\lVert \sqrt{\vert \dot{z}_\rho \vert_\rho''}  \xi_\rho \Big\rVert_{L^\infty(I,H)}
	\end{equation*}
	is bounded. Hence we can pass to the limit
	\begin{equation} \label{proof: NOB a,e 1}
	\Big\vert \langle q_\rho,\vert \dot{z}_\rho\vert \, \phi\, \eta \rangle_{L^2(I,V^*),L^2(I,V)} \Big\vert
	\leq T \sqrt{2\rho} \lVert \phi \rVert_{L^\infty(\Omega)} \cdot \lVert \eta \rVert_{L^\infty(I)} \cdot \Big\lVert \sqrt{\vert \dot{z}_\rho \vert_\rho''} \;  \xi_\rho  \Big\rVert_{L^\infty(I,H)}
	 \overset{\rho\searrow 0}{\longrightarrow} 0.
	\end{equation}
Since $q_\rho \rightharpoonup^* q$ in $L^\infty(I,V^*)$  and $\dot{z}_\rho \to \dot{z}$ in $L^2(I,V)$ we have
	\begin{equation*}
	\langle q_\rho,\vert \dot{z}_\rho\vert \, \phi\, \eta \rangle_{L^2(I,V^*),L^2(I,V)}
	\longrightarrow
	\langle q,\vert \dot{z}\vert\, \phi\, \eta \rangle_{L^2(I,V^*),L^2(I,V)},
	\end{equation*}
which proves the assertion.
\end{proof}

\begin{remark} \label{remark: weak cone a,e}
	The corresponding result in \cite[Lemma 5.1]{SWW2016} is

	\begin{equation*}
	\int\limits_{I} \langle q(t), \phi(t) \vert \dot{z}(t) \vert \rangle_{H^{-1}(\Omega),H^1_0(\Omega)}\dt = 0 \qquad \forall \phi \in \mathcal{C}_0^\infty(I\times \Omega),
	\end{equation*}
	which is weaker than our result, since we have the equality pointwise f.a.a. $t\in I$.
\end{remark}

By summarizing our previous results we obtain an optimality system for the non-smooth problem. The next theorem, which is our main result, corresponds to \cite[Theorem 5.2]{SWW2016}.

\begin{theorem} \label{theorem: optimality system nonsmooth}
	Let $(\bar{z},\bar{g})\in H^1_\star(I,V) \times H^1_\star(I,H)$ be a local solution of \eqref{eq: non-smooth problem} and $p\in (1,\infty)$. Then there exist adjoint variables  $u \in W^{1,p}(I,V^*)$ and $\xi \in L^\infty(I,V)$. Define $q := u + \sigma \Delta \xi$. Then we have
	\begin{subequations}
		\begin{align}
		&\begin{cases}
		\dot{u} &= -\Delta \xi - j_1'(\bar{z}),\\
		u(T) &= j_2'(\bar{z}(T)),
		\end{cases} \label{NOB nonsmooth1}
		\\
		&\begin{cases}
		-\ddot{\bar{g}} +\bar{g} +\xi &= 0, \\
		\bar{g}(0) &= 0, \\
		\dot{\bar{g}}(T) &= 0,
		\end{cases} \label{NOB nonsmooth2}
		\\
		&\langle q, \phi \vert \dot{\bar{z}}\vert \rangle_{V^*,V} = 0 \qquad \forall \; \phi \in \mathcal{C}_0^\infty(\Omega), \label{NOB nonsmooth3}\\
		&\langle u(t), \xi(t) \rangle_{V^*,V} \geq   \sigma \lVert \xi(t) \rVert_V^2 \qquad \text{f.a.a.\@ } t\in I. \label{NOB nonsmooth4}\\
		&\langle q(t), \xi(t) \rangle_{V^*,V} \geq 0 \qquad \text{f.a.a.\@ } t\in I. \label{NOB nonsmooth5}
		\end{align}
	\end{subequations}
	The system \eqref{NOB nonsmooth1} has to be understood as an initial value problem in the Banach space $V^*$, which is equivalent to the equation
	\begin{equation} \label{NOB nonsmooth system 1}
	u(t) = j_2'(\bar{z}(T)) + \int\limits_{t}^{T} \Delta \xi(s)\ds + \int\limits_{t}^{T} \big[j_1'(\bar{z})\big](s)\ds \quad \text{in } V^* \quad \text{for a.a.}\; t\in I.
	\end{equation}
	The system \eqref{NOB nonsmooth2} holds in the following weak sense:
	\begin{equation} \label{NOB nonsmooth system 2}
	\bigl( \xi, h \bigr)_{L^2(I,H)} + \bigl(\bar{g},h\bigr)_{H^1(I,H)} = 0 \qquad \forall \, h \in H^1_\star(I,H).
	\end{equation}
\end{theorem}

\begin{proof}
Let $\rho>0 $ and $(z_\rho,g_\rho)\in H^1_\star(I,V)\times H^1_\star(I,H)$ be a global solution of \eqref{eq: smooth optimal control problem} such that $z_\rho \to \bar{z}$ in $H^1_\star(I,V)$ and $g_\rho \to g$ in $H^1_\star(I,H)$, c.f.  Theorem \ref{theorem: convergence solutions}. Due to the boundedness properties in Corollary \ref{cor: Beschränktheit Adjungiert} we can choose weak- or weak*-convergent subsequences and pass to the limit in \eqref{NOB nonsmooth system 1} and \eqref{NOB nonsmooth system 2}. Condition \eqref{NOB nonsmooth3} holds due to Lemma \ref{lemma: optimality condition1}.\\
In order to prove \eqref{NOB nonsmooth4} we test \eqref{NOB glatt c} with $\xi_\rho$ and obtain f.a.a. $t\in I$

\begin{equation*}
\langle u_\rho(t), \xi_\rho(t)\rangle_{V^*,V}  =  \sigma \lVert \xi_\rho(t) \rVert_V^2 + \int\limits_\Omega \vert \dot{z}_\rho(t) \vert_\rho'' \xi_\rho(t)^2 \; \dx \geq \sigma \lVert \xi_\rho(t) \rVert_V^2.
\end{equation*}

The boundedness properties of $u_\rho, \xi_\rho$ and embeddings imply $u_\rho \to u$ in $C(\bar{I},V^*)$ and $\xi_\rho \rightharpoonup \xi$ in $L^1(I,V)$. Hence $\langle u_\rho , \xi_\rho \rangle_{V^*,V} \to \langle u , \xi \rangle_{V^*,V}$ in $L^1(I)$ and we can choose a subsequence, which converges pointwise f.a.a. $t\in I$.\\
Finally, equation \eqref{NOB nonsmooth5} can be proven by testing the equation $q=u+\sigma \Delta \xi$ with $\xi$ and using  \eqref{NOB nonsmooth4}:

\begin{equation*}
\langle q(t), \xi(t)\rangle_{V^*,V}   = \langle u(t), \xi(t)\rangle_{V^*,V}   - \sigma \lVert \xi(t) \rVert_V^2 \geq 0.
\end{equation*}

\end{proof}

Let us compare the previous result to the formal optimality conditions of section \ref{chapter: Formal derivation optimality system}.
The equations \eqref{NOB nonsmooth1}--\eqref{NOB nonsmooth2} are equal to \eqref{eq51a}--\eqref{eq51d}.
The relation \eqref{NOB nonsmooth3} is a weak formulation of the formal conditions \eqref{formal optimality system cone a}, \eqref{formal optimality system cone e}. Moreover, the formal condition \eqref{formal optimality system cone} implies $q\xi\ge0$, which corresponds to \eqref{NOB nonsmooth5}.
It is an open problem whether all the other implications of \eqref{formal optimality system cone} can be proven
to be necessary for local optimality.

\begin{remark}\label{compare_to_SWW}
 Theorem \ref{theorem: optimality system nonsmooth} gives stronger results than \cite[Theorem 5.2]{SWW2016}.
 First of all, due to the presence of the posivitive viscosity parameter $\sigma$, we were able to proof the regularity
 $u \in W^{1,p}(I,V^*)$ and $\xi \in L^\infty(I,V)$,
 which is stronger than the regularity obtained in \cite{SWW2016}:
 $u \in L^\infty(I,V^*)$ and $\xi \in W^{-1,p}(I,L^2(\Omega))$.
 In addition, the non-negativity conditions \eqref{NOB nonsmooth3} and  \eqref{NOB nonsmooth4} are new.
 Similarly as above, one can argue that these conditions are also valid for the problem considered in \cite{SWW2016}.
\end{remark}

\subsection{Towards additional complementarity conditions}
\label{section: Further optimality conditions}

The conditions \eqref{formal optimality system cone b}-\eqref{formal optimality system cone d} remain unproven. In this section we turn our focus on the condition \eqref{formal optimality system cone c}, which is

\begin{equation*}
\dot{z}(t,x)=0,\quad \vert g(t,x) + \Delta z(t,x) + \sigma \Delta \dot{z}(t,x)\vert < 1 \quad \Longrightarrow \quad \xi(t,x)= 0.
\end{equation*}

We will now present a possible proof for this optimality condition, which however requires a strong assumption about the sequence $\vert \dot{z}_\rho \vert_\rho'$. From now on we work with the function $\vert \cdot \vert_\rho$, which was suggested in the end of Subsection \ref{subsection: Approximation of absolute value}.

\begin{equation}
\vert v \vert_\rho =
\begin{cases}
\vert v \vert & \vert v \vert > \rho, \\
\frac{1}{3}\rho + \frac{1}{\rho^2}v^2(\rho - \frac{1}{3}\vert v \vert) & \vert v \vert \leq \rho,
\end{cases}.
\end{equation}
\label{explizite approx funct}

\begin{lemma}
	Let $(\bar{z},\bar{g})\in H^1_\star(I,V) \times H^1_\star(I,H)$ be a local solution of \eqref{eq: non-smooth problem}. Assume that there exists a sequence $(\rho_n)_n$ with $\rho_n \searrow 0$ such that  $\vert \dot{z}_n(t,x)\vert_{\rho_n}'$ is pointwise convergent for a.a. $(t,x)\in I\times \Omega$.\\
	Then there exists adjoint variables $(u, \xi)$ as in Theorem \ref{theorem: optimality system nonsmooth} such that
	$\xi(t,x)= 0$ for almost all $(t,x)\in I\times \Omega$ that satisfy $\dot{\bar{z}}(t,x)= 0$ and $ \vert \bar{g}(t,x) + \Delta \bar{z}(t,x) + \sigma \Delta \dot{\bar{z}}(t,x)\vert < 1$.
\end{lemma}

\begin{proof}

	Define the set

	\begin{equation*}
	M := \{ x\in \Omega \mid \dot{z}(t,x) = 0, \; \vert \bar{g}(t,x) + \Delta \bar{z}(t,x) + \sigma \Delta \dot{\bar{z}}(t,x)  \vert <1 \},
	\end{equation*}
	which is well-defined up to a set of zero measure.
	Let $(\rho_n)_n \in \R $ be a sequence with $ \rho_n \searrow 0 $ and $(z_n,g_n)$ the solution of $(P_{\rho_n})$ with
	\begin{align*}
	&g_n \to \bar{g} \qquad \text{in }  H^1_\star(I,H), \\
	&z_n \to \bar{z} \qquad \text{in} \; H^1_\star(I,V),
	\end{align*}
	c.f. Theorem \ref{theorem: convergence solutions}.
	Hence, it holds
	$g_n\to \bar{g},\; \Delta z_n \to \Delta \bar{z},\; \Delta \dot{z}_n\to \Delta \dot{\bar{z}}$ each in $L^2(I,V^*)$.
	Moreover, $\vert \dot{z}_\rho\vert_\rho'$ converges to $\bar{g} + \Delta \bar{z} + \sigma \Delta \dot{\bar{z}}$ in $L^2(I,V^*)$ since
	\begin{align*}
	\big\lVert \vert \dot{z}_n \vert_{\rho_n}' - \bigl[  \bar{g} + \Delta \bar{z}+ \sigma \Delta \dot{\bar{z}} \bigr] \big\rVert_{L^2(I,V^*)}
	&= \big\lVert \bigl[ g_\rho + \Delta z_\rho + \sigma\dot{z}_\rho \bigr] - \bigl[  \bar{g} + \Delta \bar{z}+ \sigma \Delta \dot{\bar{z}} \bigr] \big\rVert_{L^2(I,V^*)} \\
	&\longrightarrow 0 \qquad \text{for}\; \rho \searrow 0.
	\end{align*}
	By assumption, $\vert \dot{z}_n(t,x) \vert_{\rho_n}'$ is pointwise convergent a.e.\@ on $I\times \Omega$.
	Since $\vert v \vert_{\rho_n}'\in [-1,1]$ for all $v\in \R$ we obtain by Lebesgue's dominated convergence theorem
	that $\vert \dot{z}_n \vert_{\rho_n}'$ converges in $L^2(I\times \Omega)$ to the pointwise limit. Hence $\bar{g}+ \Delta \bar{z}+ \sigma \Delta \dot{\bar{z}}$ is the pointwise limit of $\vert \dot{z}_n \vert_{\rho_n}'$.

	Now let $0 <\eps < \frac{1}{2}$. We define the family of sets
	\begin{equation*}
	M_\eps := \{ (t,x)\in I\times \Omega \mid \dot{z}(t,x) = 0, \; \vert \bar{g}(t,x) + \Delta \bar{z}(t,x) + \sigma \Delta \dot{\bar{z}}(t,x) \vert <1-2 \eps \}.
	\end{equation*}
Due to Egorov's Theorem, there exists  for all $\delta >0$  a set $B_\delta \subset I\times \Omega$ such that $\vert \dot{z}_\rho \vert_\rho'$ converges uniformly on $B_\delta$ and $\big\vert \bigl( I\times \Omega \bigr) \setminus B_\delta \big\vert \leq \delta $.
In particular, it follows $\vert M_\eps \setminus B_\delta \vert \leq \delta$.
	The pointwise convergence of $\vert \dot{z}_n \vert_{\rho_n}'$ implies
	\begin{equation*}
	\lim\limits_{n\to \infty} \vert \dot{z}_n(t,x) \vert_{\rho_n}' \leq 1- 2\eps \qquad \text{a.e.\@ on } M_\eps,
	\end{equation*}
	and due to the uniform convergence on $B_\delta$ there exists $N_0\in \N$ such that
	\begin{equation} \label{proof: lemma cone c 1}
	\vert \dot{z_n}(t,x) \vert \leq 1 - \eps \qquad \text{for a.a.}\; (t,x) \in M_\eps \cap B_\delta \quad\text{and}\; \forall\; n \geq N_0.
	\end{equation}
	Now recall that we made a particular choice of $\vert \cdot \vert_{\rho_n}'$, see page \eqref{explizite approx funct}. It is easy to verify that this function satisfies
	\begin{equation*}
	\big\vert \vert v \vert_{\rho_n}' \big\vert \leq 1-\eps \quad \Leftrightarrow \quad \vert v \vert \leq \rho_n(1-\sqrt{\eps}) \quad
	\Leftrightarrow \quad \vert v \vert_{\rho_n}'' \geq \frac{2}{\rho_n} \sqrt{\eps}.
	\end{equation*}
	Hence, it holds
	\begin{equation*}
	\vert \dot{z}_n \vert_{\rho_n}'' \geq \frac{2}{\rho_n} \sqrt{\eps} \qquad \text{for a.a.}\; (t,x) \in M_\eps \cap B_\delta, \quad\text{and}\; \forall\; n \geq N_0.
	\end{equation*}
The boundedness of $\Big\lVert  \vert  \dot{z}_n \vert_{\rho_n}'' \xi_n^2 \Big\rVert_{L^1(I,L^1(\Omega))}$, cf., Corollary \ref{cor: Beschränktheit Adjungiert}, implies the existence of a $C>0$ such that for all $n\geq N_0$
	\begin{align*}
	C \geq  &\lVert \vert \dot{z}_n(t,x) \vert_{\rho_n}'' \xi_n(t,x)^2 \rVert_{L^1(I,L^1(\Omega))}
	\geq \int_{M_\eps \cap B_\delta} \vert \dot{z}_n (t,x)\vert_{\rho_n}'' \xi_n(t,x)^2\; \dx \dt \\
	&\geq \frac{2\sqrt{\eps}}{\rho_n} \int_{M_\eps \cap B_\delta} \xi_n(t,x)^2\; \dx \dt
	\end{align*}
is satisfied.
	Due to embedding theorems for Sobolev functions there exists $p>2$ such that $V\hookrightarrow L^p(\Omega)$. Choose $q$ such that $\frac{1}{p} + \frac{1}{q} = \frac{1}{2}$. We obtain
	\begin{equation*}
	0 \leq \int_{M_\eps} \xi_n^2 \; \dx\dt
	\leq \frac{\rho_n}{2\sqrt{\eps}} C + \int_{M_\eps \setminus B_\delta} \xi_n^2\; \dx\dt
	\leq \frac{\rho_n}{2\sqrt{\eps}} C + \delta^{\frac{1}{q}} \lVert \xi_n\rVert_{L^p(I,V)}.
	\end{equation*}
	As $\xi_n$ is bounded in $L^\infty(I,V)$ by Corollary \ref{cor: Beschränktheit Adjungiert}, there is a constant $C'>0$ such that
	\[
	0 \leq \int_{M_\eps} \xi_n^2 \; \dx\dt \le C' \left( \frac{\rho_n}{\sqrt{\eps}}  + \delta^{\frac{1}{q}}\right).
	\]
	These previous results yields
	\begin{equation*}
	0\le \limsup\limits_{n\to \infty} \int_{M_\eps} \xi_n^2\;\dt \dx
	\leq \limsup\limits_{n\to \infty}C' \left( \frac{\rho_n}{\sqrt{\eps}}  + \delta^{\frac{1}{q}}\right)
	=C' \delta^{\frac{1}{q}}.
	\end{equation*}
	Since $\delta>0$ was arbitrary, we can conclude $\lim\limits_{n\to \infty} \int_{M_\eps} \xi_n^2\; \dt \dx = 0$. Hence,
	 we showed $\xi_n \to 0$ in $L^2(M_\eps)$.
	 Due to the boundedness of $\xi$ in $L^2(I,H)$, we get $\xi_n \rightharpoonup \xi$ (for a subsequence). Weak and strong limits have to be the same, and therefore it follows $\xi = 0$ a.e.\@ on $M_\eps$ (for all weak subsequential limit points of $\xi_n$).
Since $M = \bigcup_{k\in \N} M_\frac{1}{k}$, we obtain $\xi = 0$ a.e.\@ on $M$.
\end{proof}

In the next corollary we give a sufficient condition for pointwise convergence of $\vert \dot{z}_n \vert_\rho'$.

\begin{corollary}
	Let $p \in (1,\infty)$ and assume that there exists a sequence $(\rho_n)_n$ with $\rho_n \searrow 0$ such that $\vert \dot{z}_n \vert_{\rho_n}'$ is bounded in $L^p(I,V)$. Then there exists a subsequence (denoted again by $\rho_n$) such that $\vert \dot{z}_n \vert_{\rho_n}'$ is pointwise convergent a.e.\@ on $I\times \Omega$.
\end{corollary}

\begin{proof}
	Since $z_n$ solves the smooth state equation we have
	\begin{equation*}
	\vert \dot{z}_n(t) \vert_{\rho_n}' = g_n + \sigma \Delta \dot{z}_n + \Delta z_n \qquad \text{a.e.\@ on } I.
	\end{equation*}
	We obtain by the boundedness of $g_n$ and $z_n$ that $\bigl( \vert \dot{z}_n \vert_{\rho_n}' \bigr)_n$ is bounded in $H^1(I,V^*)$.
	Hence, $\bigl( \vert \dot{z}_n \vert_{\rho_n}' \bigr)_n$ is bounded in $L^p(I,V)\cap H^1(I,V^*)$. Due to  the Aubin-Lions lemma \cite{A1963,L1969}, the embedding
	$L^p(I,V)\cap H^1(I,V^*) \hookrightarrow L^p(I,H)$.
	is compact, which proves the claim.
\end{proof}

\begin{remark}
Since $\nabla \bigl( \vert \dot{z}_n(t) \vert_\rho' \bigr) = \vert \dot{z}_n(t) \vert_\rho'' \nabla \dot{z}_n(t) $ and $\vert \cdot \vert_\rho''$ is not bounded for $\rho \searrow 0$ the sequence $\vert \dot{z}_n \vert_{\rho_n}'$ is not necessarily bounded in $L^p(I,V)$ and has to be assumed.
\end{remark}

\section{Conclusion and outlook}

We derived and proved optimality conditions for the non-smooth optimal control problem. Our optimality system is similar to that in \cite{SWW2016}. We obtained stronger results, e.g., higher regularity of the adjoint variables.
Despite the high regularity,
we were not able to prove some of the expected optimality condition.
Here, we presented an additional assumption to prove one of the missing conditions.

\bibliography{mybib}
\bibliographystyle{plain}

\end{document}